\documentclass[a4paper,11pt]{amsart}
\usepackage{comment}
\usepackage{amscd}
\usepackage{amssymb}
\usepackage{epsf,latexsym,graphicx}
\usepackage[matrix,arrow,tips,curve]{xy}

\usepackage{color}
\usepackage{tikz}
\newlength{\defbaselineskip} 
\usepackage{footnote}
\usepackage{color}
\usepackage{todonotes}
\usepackage{MnSymbol}

\usepackage{amsfonts,amssymb,amscd}
\newtheorem{thm}{Theorem}[section]
\newtheorem{cor}[thm]{Corollary}

\newtheorem{lemma}[thm]{Lemma}
\newtheorem{lem}[thm]{Lemma}
\newtheorem{prop}[thm]{Proposition}

\newtheorem{prob}[thm]{Problem}

\theoremstyle{definition}
\newtheorem{defi}[thm]{Definition}

\newtheorem{rem}[thm]{Remark}

\usepackage{fullpage}
\usepackage{tikz,tikz-cd}
\usetikzlibrary{matrix,arrows}
\makeatletter
\tikzset{
  edge node/.code={%
      \expandafter\def\expandafter\tikz@tonodes\expandafter{\tikz@tonodes #1}}}
\makeatother
\tikzset{
  subseteq/.style={
    draw=none,
    edge node={node [sloped, allow upside down, auto=false]{$\subseteq$}}},
  Subseteq/.style={
    draw=none,
    every to/.append style={
      edge node={node [sloped, allow upside down, auto=false]{$\subseteq$}}}
  }
}

 \numberwithin{equation}{section}
\numberwithin{equation}{section} \theoremstyle{definition}

\DeclareMathOperator{\Hom}{Hom}

\DeclareMathOperator{\rank}{rank}
\DeclareMathOperator{\corank}{corank}

          \newcommand\PP{{\mathbb{P}}}
           \newcommand\QQ{{\mathbb{Q}}}
           \newcommand\PL{{\mathrm{P}}}
           
          \newcommand\C{{\mathbb{C}  }}

 \newcommand\F{{\mathcal F}}
            
           \newcommand\G{{\mathcal G}}

          \newcommand\oo{\mathcal O}
          \newcommand\Z{\mathbb{Z}}

          \newcommand\rk{\mathrm{rk}}

\definecolor{zielony}{rgb}{0.5, 0.9, 0.1}
\definecolor{czerwony}{rgb}{0.8, 0.2, 0.1}
\definecolor{niebieski}{rgb}{0.3, 0.1, 0.9}

\newcounter{appendice}

\begin{document}

\title{Hyperk\"ahler fourfolds and Kummer surfaces} 

\author[A.~Iliev]{Atanas Iliev}
\address{Seoul National University, Department of Mathematics, Gwanak Campus, Bldg. 27, Seoul 151-747, Korea}
\email{ailiev@snu.ac.kr}
\author[G.~Kapustka]{Grzegorz Kapustka}
\address{Institut f\"ur Mathematik Mathematisch-naturwissenschaftliche Fakult\"at
Universit\"at Z\"urich, Winterthurerstrasse 190, CH-8057 Z\"urich}
\address{Institute of Mathematics of the Jagiellonian University, ul \L ojasiewicza 6 Krak\'ow, Poland}

\email{grzegorz.kapustka@uj.edu.pl}
\author[M.~Kapustka]{Micha\l{} Kapustka}
\address{University of  Stavanger, Department of Mathematics and Natural Sciences, NO-4036 Stavanger, Norway}
\address{Institute of Mathematics of the Jagiellonian University, ul \L ojasiewicza 6 Krak\'ow, Poland}
\email{michal.kapustka@uis.no}
\author[K.~Ranestad]{Kristian Ranestad}
\address{University of Oslo, Department of Mathematics, PO Box 1053, Blindern, N-0316 Oslo, Norway}
\email{ranestad@math.uio.no}

\begin{abstract}
We show that a Hilbert scheme of conics on a Fano fourfold double cover of $\PP^2\times\PP^2$ ramified along a divisor of bidegree $(2,2)$ admits a $\PP^1$-fibration with base being
a hyper-K\"{a}hler fourfold. We investigate the geometry of such fourfolds relating them with degenerated EPW cubes, with elements in the Brauer groups of $K3$ surfaces
of degree $2$, and with Verra threefolds studied in \cite{Vera}. 
These hyper-K\"{a}hler fourfolds admit natural involutions and complete the classification of geometric realizations
of anti-symplectic involutions on hyper-K\"{a}hler $4$-folds of type $K3^{[2]}$.

As a consequence we present also three constructions of quartic Kummer surfaces in $\PP^3$: as Lagrangian and symmetric degeneracy loci and as the base of a fibration of conics in certain threefold quadric bundles over $\PP^1$.
\end{abstract}




\maketitle
\tableofcontents
By a hyper-K\"{a}hler manifold or equivalently by an irreducible holomorphic symplectic  (or IHS) $2n$-fold we mean
a $2n$-dimensional simply connected compact K\"{a}hler manifold with trivial canonical
bundle that admits a unique (up to a constant)
non-degenerate holomorphic $2$-form (called the symplectic form) and is not a product of two
manifolds \cite{Beauville}. 
In this paper we are studying the geometry of some families of IHS fourfolds that are deformation equivalent to the Hilbert scheme of two points on a $K3$ surface (of type $K3^{[2]}$).
 
Recall from \cite{BD} that Hilbert schemes of lines on smooth cubic hypersurfaces in $\PP^5$ are IHS fourfolds of  type $K3^{[2]}$
characterized by the fact that they admit a polarization of Beauville degree $q=6$ (i.e degree $3*36$).
In \cite{OgradyEPW} O'Grady described the complete family of polarized IHS fourfolds of $K3^{[2]} $ type with Beauville degree $q=2$. He found out that such manifolds are double covers of sextic hypersurfaces defined as Lagrangian degeneracy loci. 
Next \cite{IM} described constructions of IHS fourfolds with $q=2$ as bases of $\PP^1$ fibrations
 on Hilbert schemes of conics on Fano fourfolds of degree $10$.

The aim of this article is to investigate a special $19$-dimensional family $\mathcal{U}$ of IHS fourfolds of type $K3^{[2]}$ admitting a polarization of Beauville degree $q=4$ (i.e degree $48$). In fact, the family $\mathcal{U}$ represents a component of the hyperelliptic locus in the moduli space of all  IHS fourfolds of type $K3^{[2]}$ admitting a polarization of Beauville degree $q=4$. The elements of the family $\mathcal{U}$ are obtained as double covers of some special Lagrangian degeneracy loci on a cone over $\mathbb{P}^2\times \mathbb{P}^2$.  The same family $\mathcal{U}$ is obtained by considering for a general  Fano fourfold $Y$ being the double cover of $\PP^2\times \PP^2$ branched along a bi-degree $(2,2)$ divisor (we call such $Y$ Verra fourfolds)  the Hilbert scheme $F(Y)$ of conics on $Y$.
We show that a general fivefold $F(Y)$ admits a natural $\PP^1$ fibration such that its base is an IHS fourfold in $\mathcal{U}$. Finally, we show also that the generic element from $\mathcal{U}$ is a moduli space of twisted sheaves on a $K3$ surface.

The IHS fourfolds from $\mathcal{U}$ appear naturally in the following context: 
Recall that van Geemen classified two torsion elements in the Brauer group $Br(S)$ of a general $K3$ surface $S$ that admits a polarization of degree $2$, \cite{vG}. He showed that there are three types of elements in $Br(S)_2\simeq (\Z_2)^{21}$ and that they give rise to three type of varieties $Y_{\alpha_i}$ for $i=1,2,3$ respectively:
\begin{itemize}
\item  a smooth complete intersection of three quadrics in $\PP^5$, or
\item  a cubic fourfold containing a plane, or
\item  a double cover of $\PP^2 \times \PP^2$ ramified along a hypersurface of bi-degree $(2, 2);$
\end{itemize}
such that a twist of the polarized Hodge structure defined by $\alpha_i$ is Hodge isometric to a primitive sublattice of the middle cohomology of $Y_{\alpha_i}$ for $i=1,2,3$.
There are direct geometric constructions relating $(S,\alpha_i)$ with the variety $Y_{\alpha_i}$. In the first case Mukai \cite{Mu2} showed that a moduli space of bundles on $Y_{\alpha_1}$ is isomorphic
to $S$. In 
 \cite{Bh} it is shown that $Y_{\alpha_1}$ is isomorphic 
to the moduli space of certain orthogonal bundles on $S$; giving the relation in the other direction.
Note, however, that the twist is not apparent in these construction.
One may ask whether the  $K3$ surface $Y_{\alpha_1}$ of degree $8$ is isomorphic to a moduli
space of twisted sheaves  on
$K3$ surfaces of degree 2 with the twist $\alpha_1$  \cite[\S 1]{MSTV}.


In the second case for $(S,\alpha_2)$ a geometric relation was described in \cite{MS}.
It was shown that a moduli space of twisted sheaves on $(S,\alpha_2)$ is birational to the IHS fourfold being the Hilbert scheme of lines
on a cubic fourfold containing a plane.
Our construction completes the picture by showing that the moduli space of twisted sheaves on $(S,\alpha_3)$
is isomorphic to an IHS fourfold from $\mathcal{U}$ i.e. is constructed from the Hilbert scheme of conics on the corresponding fourfold $Y_{\alpha_3}$.

\subsection{Construction via Lagrangian Degeneracy loci}
Section \ref{section2} is devoted to the construction of elements of $\mathcal{U}$ as double covers of appropriate Lagrangian degeneracy loci inside a cone $C(\PP^2\times \PP^2)\subset \PP^9$ over the Segre embedding of $\PP^2\times \PP^2$. This construction is analogous to the construction of EPW sextics  \cite{OgradyEPW}, \cite{EPW}. It is also naturally related to special EPW cubes \cite{EPWcubes}. Let us be more precise: 
Let $U_1, U_2$ be $3$-dimensional vector spaces with fixed volume forms. Consider the cone over the Segre embedding of $\PP(\wedge^2 U_1 )\times \PP(U_2) $
\[
C_{U_1}:=C(\PP(\wedge^2 U_1 )\times \PP(U_2))
\]
 interpreted as a subset 
\[
C_{U_1}=\mathrm{G}(3,U_1\oplus U_2)\cap \PP(\wedge^3 U_1\oplus (\wedge^2 U_1\otimes U_2))\subset \PP(\wedge^3(U_1\oplus U_2)).
\]
Note that we use the notation $\mathbb{P}(B)$ to denote the space of 1-dimensional subspaces of $B$. 
Consider the vector subspace $(\wedge^2 U_1\otimes U_2)\oplus ( U_1\otimes \wedge^2 U_2) \subset \wedge^3 (U_1\oplus U_2)$ equipped with a symplectic form corresponding to wedge product. Each point $[U]$ of the cone $C_{U_1}$ corresponds to a three-space $U\subset U_1\oplus U_2$ such that $\dim (U\cap U_1)\geq 2$. To $U$ we associate the Lagrangian subspace 
\begin{align*}
\bar{T}_U:=(\wedge^2 U\wedge (U_1\oplus U_2))/\wedge^3U_1\subset &(\wedge^3 U_1\oplus (\wedge^2 U_1\otimes U_2)\oplus ( U_1\otimes \wedge^2 U_2))/\wedge^3U_1\\
&\cong(\wedge^2 U_1\otimes U_2)\oplus ( U_1\otimes \wedge^2 U_2).
\end{align*}
Let $\bar{A}\subset (\wedge^2 U_1\otimes U_2)\oplus ( U_1\otimes \wedge^2 U_2)$ be a general Lagrangian subspace.
 To this subspace $\bar{A}$ we can associate degeneracy loci for each $k>0$:
$$D^{\bar{A}}_k=\{ [U]\in C_{U_1}| \dim (\bar{T}_U \cap \bar{A})\geq k\}.$$
The variety $D^{\bar{A}}_1$ is a special quartic section of $C_{U_1}$ that we call an EPW quartic section (abusing the name of the first degeneracy locus in $\mathrm{G}(3,6)$ considered in \cite{DK}). We shall prove that for a generic choice of $\bar{A}$ the fourfold $D^{\bar{A}}_1$ is singular exactly along the surface $D^{\bar{A}}_2\subset \PP(\wedge^3(U_1\oplus U_2))$ which has degree $72$.
The main result of the above construction is the following: 
\begin{thm}\label{main}  For a generic choice of $\bar{A}\in \mathrm{LG}(9,(\wedge^2 U_1\otimes U_2)\oplus ( U_1\otimes \wedge^2 U_2))$ there exists a natural double cover $X_{\bar{A}}\to D^{\bar{A}}_1$ branched along $D^{\bar{A}}_2$ such that 
$X_{\bar{A}}$ is an IHS fourfold of $K3^{[2]}$ type that admits a polarization of Beauville degree $q=4$. 
\end{thm}
The proof is presented in Section \ref{section2}. The subset $\mathcal{U}$ of the moduli space of polarized IHS fourfolds deformation equivalent to $K3^{[2]}$ and with polarization of Beauville degree 4 that parametrizes manifolds constructed in Theorem \ref{main} is of dimension $19$.
\subsection{Relation to EPW cubes} 
The construction of EPW quartic sections is more natural when seen in the context  of EPW cubes. Recall that in \cite{EPWcubes} we constructed a $20$-dimensional family (locally complete) of polarized IHS sixfolds deformation equivalent to the Hilbert scheme of three points on a $K3$-surface (i.e.~of type $K3^{[3]}$) and admitting a  polarization of Beauville degree $q=4$. The elements of this family are natural double covers of special codimension $3$ subvarieties of the Grassmannian $\mathrm{G}(3,6)$ that we called EPW cubes. The EPW quartic sections can be seen as subvarieties of special 
EPW cubes.
Recall that for a Lagrangian subspace $A\subset \wedge^3 V_6$ we define 
$$D^2_A=\{[U]\in \PP(\wedge^3V_6)| \dim (A\cap ((\wedge^2 U) \wedge V_6))\geq 2\}.$$
When $A$ is general $D^2_A$ is called an EPW cube.
If now $A\subset \wedge^3 V_6 $ is a general Lagrangian subspace that contains $\wedge^3 U_1$,   for some $U_1\subset V_6$ of dimension 3 
then $D^2_A $ 
is a special 
EPW cube. Now for every decomposition  $V_6=U_1\oplus U_2$ we have a natural identification $C_{U_1}=C(\PP(U_1)\times \PP(\wedge^2 U_2))=T_{[U_1]}\cap \mathrm{G}(3,V_6)$, where $T_{[U_1]}$ is the projective tangent space to $\mathrm{G}(3,V_6)$ in $[U_1]$. Under this identification we have
$$D^1_{\bar{A}}=D^2_A \cap C_{U_1},$$
with $\bar{A}=A/(\wedge^3 U_1)\subset (\wedge^3 U_1)^{\perp}/(\wedge^3 U_1)$. 

 \subsection{Construction via Hilbert scheme}
Our second construction of IHS fourfolds from the family $\mathcal{U}$ is the subject of Section \ref{F-Verra}. It uses Hilbert schemes of conics on so-called Verra Fano fourfolds.
Let $U_1$ and $U_2$ be $3$-dimensional vector spaces.
We call a \emph{Verra fourfold} \cite{Vera}, \cite{Iliev} an element of the $19$-dimensional family of Fano fourfolds which is the intersection $Y$ of the cone $C(\PP( U_1)\times \PP(\wedge^2 U_2))\subset \PP(\C\oplus ( U_1 \otimes \wedge^2 U_2))$ with a quadric hypersurface $Q$. 
Equivalently $Y$ is the double cover of $\PP(U_1)\times \PP(\wedge^2 U_2)=\PP^2\times \PP^2$ branched along a divisor $Z$ of bi-degree $(2,2)$. The threefold $Z$ will be called the \emph{Verra threefold} associated to the Verra fourfold $Y$.  Note that $Z$ can be identified with the section of $Y$ by the hyperplane polar to the vertex of the cone $C(\PP( U_1)\times \PP(\wedge^2 U_2))$ via the quadric $Q$.
Verra threefolds were introduced by  A. Verra in \cite{Vera} as counterexamples to the Torelli problem for Prym varieties of unbranched double coverings  of plane sextics.

 The linear system of quadrics containing $C(\PP(U_1)\times \PP(\wedge^2 U_2))\subset \PP(\C\oplus (U_1\otimes \wedge^2 U_2))$ is then naturally isomorphic to $\PP(U_1 \otimes \wedge^2 U_2)$, via a volume form on $ U_1\otimes \wedge^2 U_2\cong U_1\otimes  U_2 ^\vee$.
The linear system of quadrics containing $Y\subset \PP^9$ is therefore naturally isomorphic to $\PP(\C\oplus (U_1 \otimes \wedge^2 U_2))$ and its dual is naturally isomorphic to $\PP(\C\oplus (\wedge^2 U_1 \otimes U_2))$. 
The fourfold $Y$ admits  two natural projections $\pi_1$ and $\pi_2$ onto  $\PP(U_1)$ and $\PP(\wedge^2 U_2)$ respectively.
We denote by $F(Y)$ the Hilbert scheme of plane conic curves on $Y$ of type $(1,1)$ i.e. conics that projects to lines by both $\pi_1$ and $\pi_2$. 

Let $[C]\in F(Y)$ be a $(1,1)$-conic on $Y$, then $C$ spans a plane $\PL_C\subset \PP(\C\oplus (U_1\otimes \wedge^2 U_2))$, and the  locus $H_C$ of quadrics containing $Y\cup \PL_C$ is a hyperplane, i.e. a point $[H_C]\in \PP(\C\oplus (\wedge^2 U_1\otimes  U_2))$ in the dual space. 
In this way we define a map
\[
\psi_{Q}: F(Y)\to \PP(\wedge^3 U_1 \oplus (\wedge^2 U_1 \otimes U_2)); \qquad [C]\mapsto [H_C]
\]
We identify the image of this map in the following way.
Note that  the quadric hypersurface $Q$, such that $Y=Q\cap C(\PP( U_1)\times \PP(\wedge^2 U_2))$, induces a quadric $Q'\subset \PP(U_1 \otimes \wedge^2 U_2)$ defining the branch locus $Z$ of the double cover $Y\to \PP(U_1)\times \PP(\wedge^2 U_2)$ via $Z=Q'\cap (\PP( U_1) \otimes \PP(\wedge^2 U_2))$. The quadric $Q'$ is defined by a symmetric linear map $q'\colon ( U_1\otimes \wedge^2 U_2) \to  (\wedge^2 U_1\otimes U_2)$. The graph of such a symmetric map $q'$ is a  Lagrangian subspace that we denote $\bar{A}_Q\subset (\wedge^2 U_1\otimes U_2)\oplus ( U_1\otimes \wedge^2 U_2)$.
We shall prove that the image $\psi_{Q}(F(Y))$ coincides with the first degeneracy locus $$D^{\bar{A}_Q}_1\subset C(\PP(\wedge^2 U_1)\times \PP(U_2)).$$
Furthermore by studying fibers of the map we obtain a factorization $\psi_{Q}
= \rho\circ\phi$ with $\phi$ a $\PP^1$ fibration and $\rho$ a 2:1 map branched exactly in $D^{\bar{A}_Q}_2$.
Combining this with Theorem \ref{main} we obtain:
\begin{thm}\label{main2}  The Hilbert scheme of conics on a general Verra fourfold $Y=Q\cap C(\PP( U_1)\times \PP(\wedge^2 U_2))$ admits a $\PP^1$-fibration (a smooth map whose all fibers are isomorphic to $\mathbb{P}^1$) over the IHS fourfold $X_{\bar{A}_Q}\in \mathcal{U}$. Moreover, a general IHS fourfold $X\in \mathcal{U}$ appears in this way.
\end{thm}
As a consequence of the proof of Theorem \ref{main2} we observe furthermore that in the above notation the surface $D^{\bar{A}_Q}_2$
is on one hand isomorphic to the fixed locus of an antisymplectic involution on the IHS fourfold $X_{\bar{A}_Q}$ and on the other it admits an \'etale double cover by the Hilbert scheme of conics on the Verra threefold $Z$ (see Proposition \ref{Z_Q}).

\subsection{Moduli space of twisted sheaves}
In Section \ref{known} we show a further alternative construction of the elements of $\mathcal{U}$: as moduli spaces of twisted sheaves \cite{Yo} 
on $K3$ surfaces. More precisely we prove:
\begin{thm}\label{0.3} A general fourfold $X\in \mathcal{U}$  is isomorphic to the moduli space of stable twisted sheaves on a polarized $K3$ surface of degree $2$ with a two-torsion Brauer element. 
\end{thm}

 \subsection{Properties} Our main motivation to study the family $\mathcal{U}$ is to understand the relation between the three geometric constructions considered. 
As a result we present relations of different points of view: Hodge-theoretic, moduli-theoretic, geometric, and arithmetic.
 In particular we prove, that the generic element of $\mathcal{U}$ has Picard group of rank $2$ does not admit any polarization of Beauville degree $2$ and is not isomorphic to a moduli space of sheaves on a $K3$ surface. 
  Moreover, each element of the family $\mathcal{U}$ admits two Lagrangian fibrations and is a $8:1$ ramified cover of $\PP^2\times \PP^2$.

  In section \ref{section2} we also discuss our construction in the context of the classification of automorphisms of IHS fourfolds of type $K3^{[2]}$. In particular, we shall see that $\mathcal{U}$ is the unique $19$-dimensional irreducible family of IHS fourfolds of type $K3^{[2]}$ that is not in the closure of the family of double EPW sextics, such that each element admits an anti-symplectic involution \cite{OW}. In particular, the family $\mathcal{U}$ can be seen as a component of the hyperelliptic locus of the moduli space of  polarized  IHS fourfolds of type $K3^{[2]}$ with $q=4$. Indeed, for a general IHS fourfold of type $K3^{[2]}$ with polarization of Beauville degree $q=4$ the map defined by the polarization is birational.   The following remains a challenge:
  \begin{prob} Describe the generic polarized IHS fourfold of type $K3^{[2]}$ of Beauville degree $q=4$. 
\end{prob}

The description as  double covers of Lagrangian degeneracy loci can also be applied to study degenerations of the family $\mathcal{U}$ and permit to complete the classification of geometric realizations of automorphisms of IHS of type $K3^{[2]}$ given in \cite{MongardiWandel}.
 Note that as a direct consequence from \cite[\S 5.1]{MongardiWandel} we obtain the following: 
\begin{cor}\label{MongWand} Any IHS fourfold $X$ of type $K3^{[2]}$ that admits non-symplectic automorphism of prime order  $p\neq 3,23$ is either in the closure of the family of double EPW sextics
or in the closure of the family $\mathcal{U}$, or $X$ is isomorphic to a moduli space of stable objects on a $K3$ surface and the automorphism is induced from an automorphism of the $K3$ surface.
\end{cor}

  


Finally in section \ref{Verra-Involutions} we study the invariants of the two dimensional fixed loci of the involution on the elements from the family $\mathcal{U}$.
 Recall that Beauville studied the invariants of the fixed loci of antisymplectic involutions on IHS fourfolds in general.
In the case of $19$-dimensional families of involutions on IHS fourfolds with $b_2=23$ it follows from \cite[Theorem~2]{beau-invo}
that the invariants of the fixed locus $F$ are $K^2_F=288$ and $\chi(\oo_F)=37$. 
Using Proposition \ref{Z_Q} we are able to deduce  the invariants of a Hilbert schemes of conics on a Verra threefold $Z$. 
The computation of all invariants is included in Proposition \ref{3-2}.

\subsection{Relation to Kummer surfaces} In section \ref{kummer}, we describe a "Baby case" of our constructions by presenting two constructions of the Kummer surfaces first as Lagrangian degeneracy loci (as in \cite[Theorem~9.2]{EPW}) and next as a quotient of the base of a fibration on the Hilbert scheme of $(1,1)$-conics on a quadric section of a cone $C(\PP^1\times \PP^2)\subset \PP^6$ over the Segre embedding $\PP^1\times \PP^2\subset \PP^5$.
The relation to the description of the EPW quartic section is explained in Section \ref{lagr}. In particular, we shall see that the EPW quartic section admits two fibrations by Kummer surfaces. The descriptions of EPW quartic sections via Lagrangian degeneracy loci and Hilbert scheme fibration restrict to the obtained descriptions of Kummer surfaces.

 Furthermore, in Section \ref{kummer} we provide in addition a third construction for a general Kummer surface: as a component of the discriminant locus of the system of quadrics containing the Verra fourfold, or equivalently as the associated symmetric degeneracy locus.

{\bf Acknowledgements.} We thank G.~Mongardi for several helpful discussions, in particular for suggesting us Theorem \ref{0.3}. We thank A.~Kresch,  A.~Kuznetsov and L.~Manivel for helpful comments. A. Iliev was supported by SNU grant 0450-20130016, G. Kapustka by Iuventus plus 0301/IP3/2015/73 ``Teoria reprezentacji oraz wlasno\'{s}ci rozmaito\'{s}ci siecznych'', M. Kapustka by NCN grant 2013/10/E/ST1/00688 and K. Ranestad by RCN grant 239015.

\subsection{Notation}\label{notation}
Let us explain here some of the notation used in the paper. 
 Let   $V$  be a complex $6$-dimensional complex vector space, and
 fix an isomorphism  $vol: \wedge^6 V \to \mathbb{C}$. 
It induces a natural skew-symmetric form
\begin{equation}\label{def of eta1}\textstyle 
\eta: \wedge^3V\times \wedge^3 V\to \mathbb{C}, \quad (\omega,\omega')\mapsto vol(\omega\wedge \omega').
 \end{equation} 
 We denote by $\mathrm{LG}_\eta(10, \wedge^3V)$ the variety of $10$-dimensional isotropic (i.e. Lagrangian) subspaces of $\wedge^3W$ with respect to $\eta$.
 For any $3$-dimensional subspace $U\subset V$, the $10$-dimensional subspace 
 \[\textstyle 
 T_U:=\wedge^2U\wedge V\subset \wedge^3V
 \]
 belongs to $\mathrm{LG}_\eta(10, \wedge^3V)$, and $\PP(T_U)$ is the projective tangent space to $\mathrm{G}(3,V)\subset \PP(\wedge^3 V)$ at $[U]$.
 Therefore, the family $\{T_U\;|\; [U]\in \mathrm{G}(3,V)\}$ forms a symplectic vector bundle of rank $10$ over $\mathrm{G}(3,V)$.
 
 For any  $[A]\in \mathrm{LG}_\eta(10, \wedge^3V)$ and $k\in \mathbb{N}$, 
 we consider the following Lagrangian degeneracy locus, with natural scheme structure \cite{PragaczRatajski},
 \[
 D_k^A=\{[U]\in \mathrm{G}(3,V)\;|\; \dim (A\cap T_U)\geq k\}\subset \mathrm{G}(3,V).
 \]
The variety $D_2^A$ is an {\it EPW cube}.
In the present paper we study special EPW cubes corresponding to the choice of Lagrangian space $A\in \Sigma$, where 

$$\Sigma =\{ [A]\in \mathrm{LG}_\eta(10,\wedge^3V)\;|\;  \PP(A)\cap \mathrm{G}(3,V)\not=\emptyset \}$$ 
as in  \cite{Ogrady-michigan} and \cite{EPWcubes}.  From the same references we recall the notation for the following additional subsets of $\mathrm{LG}_\eta(10,\wedge^3V)$:
 $$\Delta=\{ [A]\in \mathrm{LG}_\eta(10,\wedge^3V)\;|\;  \exists v\in V \colon \dim A\cap (v\wedge(\wedge^2 V))\geq 3 \}, $$

$$\Gamma= \{A\in \mathrm{LG}_\eta(10,\wedge^3V)\;|\;  \exists [U]\in \mathrm{G}(3,V) \colon \dim A\cap T_U\geq 4\}.$$

 For $[U_1]\in  \PP(A)\cap \mathrm{G}(3,V)$ the Lagrangian space $A\subset \wedge^3V$ is contained in $(\wedge^3U_1)^{\bot}$, and thus defines a Lagrangian space $\bar{A}_{U_1}\subset (\wedge^3U_1)^{\bot}/ (\wedge^3{U_1})$.  Clearly  
 \[
 T_{U}\subset (\wedge^3U_1)^{\bot}\subset \wedge^3V 
 \]
  for any $[U]\in \mathrm{G}(3,V)\cap \PP(T_{U_1})$ so we define 
 \[
 D_k^{\bar{A}_{U_1}}=\{[U]\in \mathrm{G}(3,V)\cap \PP(T_{U_1})\;|\; \dim (\bar{A}_{U_1}\cap (T_{U}/{(\wedge^3U_1))})\geq k\}= \mathrm{G}(3,V)\cap\PP(T_{U_1})\cap D_{k+1}^{A}.
 \] 
  The variety $D_1^{\bar{A}_{U_1}}$ is an EPW quartic section.

Denote after O'Grady \cite {Ogrady-michigan}:
$$\tilde{\Sigma}:=\{ ([U],[A])\in \mathrm{G}(3,V)\times \mathrm{LG}(10,\wedge^3 V)\;|\; \wedge^3U\subset A \},$$
$$\tilde{\Sigma}(d):=\{ ([U],[A])\in \tilde{\Sigma} \;|\; \dim(A\cap (\wedge^2U\wedge V))\geq d+1 \},$$
$$\Theta_A:=\{[U]\in \mathrm{G}(3,V) \;|\; \wedge^3U \subset A \}.$$
$$\Sigma_{+}=\{[A]\in \Sigma \;|\; \ \ Card(\Theta_A)>1\},$$
If $\pi\colon \mathrm{G}(3,V)\times \mathrm{LG}(10,\wedge^3 V)\to \mathrm{LG}(10,\wedge^3 V)$ is the projection, then we set $\Sigma[d]:=\pi(\tilde{\Sigma}(d))$.

\section{Kummer surfaces--the first case}\label{kummer}
In this section we present a special construction of the Kummer quartic surface as a first Lagrangian degeneracy locus and at the same time as a symmetric degeneracy locus,  as well as the base of a fibration on the Hilbert scheme of conics on a Fano threefold. This shows, in particular, that the Kummer quartic can be seen as the "baby case" of the EPW sextic construction.
In the section \ref{lagr} we shall see that the Kummer quartic is a building block in the construction of our $19$-dimensional family 
$\mathcal{U}$.

\subsection{Kummer surfaces as Lagrangian degeneracy loci}\label{lagrangiankummer} Denote by $V=V_2\oplus V_4$ the complex $6$-dimensional vector space decomposed in the direct sum of a $2$-dimensional space $V_2$ and a $4$-dimensional space $V_4$.
 Set an isomorphism  $vol: \wedge^6 V=\wedge^2 V_2\otimes \wedge^4 V_4 \to \mathbb{C}$ by fixing isomorphisms  $vol_i: \wedge^i V_i \to \mathbb{C}$.
The isomorphism induces a natural skew symmetric form
\begin{equation}\label{def of eta}\textstyle 
\eta: \wedge^3V\times \wedge^3 V\to \mathbb{C}, \quad 
(\omega,\omega')\mapsto vol(\omega\wedge \omega'),
 \end{equation} 
 which restricts to a nondegenerate skew symmetric form $\eta_{2,4}$ on the $12$-dimensional subspace 
 \[
 V_{2,4}=V_2\otimes \wedge^2 V_4\subset \wedge^3 V.
 \]
 For each $v\in V_4$ the $6$-dimensional subspace 
 \[
 F_v:=V_2\otimes V_4\wedge v\subset V_{2,4}
 \]
 is Lagrangian with respect to $\eta_{2,4}$.
 Let $A\subset V_{2,4}$ be a general Lagrangian $6$-space, and let 
 \[
 D_i^A=\{[v]\in \PP(V_4)|\rank A\cap (V_2\otimes  V_4\wedge v)\geq i\}.
 \]
  \begin{lemma}\label{lagrangian Kummer} $D_1^A$ is a Kummer quartic surface singular in $D_2^A$; a set of $16$ points. 
 \end{lemma}
 \begin{proof}  Let $\mathrm{LG}(6,V_{2,4})$ denote the Lagrangian Grassmannian parameterizing the Lagrangian subspaces of  $V_{2,4}$, and let ${\mathcal F}$ be the universal rank $6$ quotient bundle on $\mathrm{LG}(6,V_{2,4})$.  The map 
 \[
 \phi: \PP(V_4)\to \mathrm{LG}(6,V_{2,4});\quad [v]\mapsto [F_v]
 \]
 is an embedding, and the pullback $\phi^*({\mathcal F})$ is a rank $6$ bundle ${\mathcal F}_{\PP(V_4)}$ on $\PP(V_4)$.
 By construction $F_v$ is a direct sum of two copies of a plane $\PP(V_4\wedge v)\subset \PP(\wedge^2 V_4)$, so ${\mathcal F}_{\PP(V_4)}$ is a direct sum of two copies of a bundle $F_0$ on $\PP(V_4)$ with total Chern class $c(F_0)=1+2h+2h^2$, where $h$ is the class of  hyperplane in $\PP(V_4)$.  Therefore ${\mathcal F}_{\PP(V_4)}$ has total Chern class
 \[
 c({\mathcal F}_{\PP(V_4)})=1+4h+8h^2+8h^3+4h^4. 
\]
The class in $\PP(V_4)$ of the degeneracy  $D_i^A$ is now the degeneracy of the natural map $\phi^*(A)\to { \mathcal F}_{\PP(V_4)}$. 
 The first bundle $\phi^*(A)$ is trivial, so, by the formulas of Pragacz and Ratajski \cite[Theorem 2.1]{PragaczRatajski}, these degeneracy classes are given by the Chern classes of ${ \mathcal F}_{\PP(V_4)}$:
\[
[D_1^A]= c_1({\mathcal F}_{\PP(V_4)})=4h,\quad [D_2^A]= (c_1c_2-2c_3)({\mathcal F}_{\PP(V_4)})=16h^3.
\]
 \end{proof}
 \begin{rem} Similarly, for any $3$-dimensional subspace $U\subset V_4$,  the subspace 
 \[
 V_2\otimes \wedge^2 U\subset V_{2,4}
 \]
  is Lagrangian with respect to $\eta_{2,4}$.  The degeneracy loci 
  \[
  \hat{D}_i^A= \{[U]\in \PP(V_4^\vee)| \rank A \cap(V_2\otimes \wedge^2 U)\geq i\}, \;(i=1,2)
  \]
   are then again a Kummer surface $\hat{D}_1^A$ and $16$  points $\hat{D}_2^A$ forming the singular locus of $\hat{D}_1^A$. 
 \end{rem}
 The Lagrangian degeneracy loci $D_i^A$,  may also be interpreted as  symmetric degeneracy loci:
 
 \subsection{Kummer surfaces as symmetric degeneracy loci}
 Fix a decomposition $V_4=\langle v_0\rangle \oplus V_3$ and the Lagrangian subspace $F_{v_0}=V_2\otimes V_4\wedge v_0\cong V_2\otimes V_3$, and let $B\subset V_{2,4}$ be a Lagrangian subspace such that $F_{v_0}\cap B=0$.
 Then $B$ is naturally isomorphic to $F_{v_0}^\vee\cong V_2^\vee\otimes \wedge^2 V_3$.  
 The Lagrangian space $A$ is then the graph in $V_{2,4}=F_{v_0}\oplus B\cong F_{v_0}\oplus F_{v_0}^\vee$ of a  linear symmetric map $F_{v_0}\to F_{v_0}^\vee$.
 Composing with the natural isomorphism $V_2\otimes V_3\to F_{v_0}$ and its transpose $F_{v_0}^\vee\to V_2^\vee\otimes \wedge^2 V_3$, we obtain a linear map
 \[
q_A: V_2\otimes V_3\to V_2^\vee\otimes \wedge^2 V_3
 \]
 inducing a symmetric bilinear form that, by abuse of notation, we shall denote by the same name 
  $$q_A: (V_2\otimes V_3)\times  (V_2\otimes V_3)\to \mathbb{C}.$$
 Denote by $Q_A=\{[\alpha]|\  q_A(\alpha ,\alpha)=0\}\subset \PP(V_2\otimes V_3)$ the quadric defined by $q_A$. Abusing  notation again $Q_A$ will also be the quadric polynomial defined by $Q_A(\alpha):=q_A(\alpha ,\alpha)$ defining the quadric $Q_A$.
 Similarly, for every $v\in V_3$ the map $(v_2\otimes v_3)\mapsto v_2\otimes  v_3\wedge v$ extends linearly to a symmetric bilinear map
  \[
 q_v:(V_2\otimes V_3) \times  (V_2\otimes V_3)\to \mathbb{C}.
 \]
Denote by $Q_v=\{[\alpha]|\ q_v(\alpha ,\alpha)=0\}\subset \PP(V_2\otimes V_3)$ the quadric defined by $q_v$ and again also the quadratic polynomial defining the quadric.
  Notice that  $Q_v$  vanishes on the Segre $3$-fold 
  \[
  \Sigma_{2,3}=\{[v_2'\otimes v_3]\in \PP(V_2\otimes V_3)|\ v_2\in V_2, v_3\in V_3\},
  \]
  and in fact 
  $[v]\mapsto Q_v$ defines an isomorphism 
  \[
  \PP(V_3)\to \PP(H^0({\mathcal I}_{\Sigma_{2,3}}(2))).
  \]
Let $S_A=\Sigma_{2,3}\cap Q_A$. Then there is similarly a natural isomorphism
\[
V_4\cong H^0({\mathcal I}_{S_A}(2));\quad  v+\lambda v_0\mapsto  q_v+\lambda q_A.
\]
 
 Let 
 \[
 {\mathcal D}_i=\{[v]\in \PP(V_4)|\corank q_v\geq i\}
 \]
 be the $i$-th degeneracy locus in $\PP(V_4)$ of the linear system of quadrics $\{Q_v|[v]\in \PP(V_4)\}$.
 Since the quadrics in the ideal of $\Sigma_{2,3}$ have rank $4$, i.e. corank $2$, we get that  ${\mathcal D}_1$ contains the plane $\PP(H^0({\mathcal I}_{\Sigma_{2,3}}(2)))$ with multiplicity $2$, and 
 ${\mathcal D}_2$ contains this plane with multiplicity $1$.  The relation between the Lagrangian loci $D_i^A$ and the symmetric loci ${\mathcal D}_i$ is described in the following:
 \begin{lemma}
 $D_i^A\cup \PP(H^0({\mathcal I}_{\Sigma_{2,3}}(2)))= {\mathcal D}_i$
 \end{lemma}
 \begin{proof} 
 It suffices to show that if $\beta =q_A(\alpha)$ and $(\alpha\wedge v_0+\beta)\in  F_{v+\lambda v_0}\cap A $ , then 
 \[
(q_v+\lambda q_A)(\alpha)=0.
 \]
 To show this we may assume that 
 \[\alpha =v_2\otimes v_3+v_2'\otimes v_3'\in V_2\otimes V_3
 \]
  and let
 \[
 q_A(\alpha)=\beta=v_2\otimes \beta_1+v_2'\otimes \beta_2
 \]
  with $\beta_i\in \wedge^2V_3$.  Then
 \[
 v_0\wedge \alpha+\beta\in A\in F_{v+\lambda v_0}\cap A\quad {\rm iff}\quad  (v_0\wedge \alpha+\beta)\wedge (v+\lambda v_0)=0.
 \]
 The right hand side is
 \begin{eqnarray}\nonumber
 &v_0\wedge \alpha\wedge v+\beta\wedge (v+\lambda v_0)=\\
 \nonumber
 &v_2\otimes (v_3\wedge v\wedge v_0 +\beta_1\wedge v+\lambda \beta_1\wedge v_0)  +v_2'\otimes (v_3'\wedge v\wedge v_0 +\beta_2\wedge v+\lambda \beta_2\wedge v_0) =0
 \end{eqnarray}
 and is equivalent to 
 \[
  \beta_1\wedge v= \beta_2\wedge v=0\quad {\rm and}\quad \lambda  \beta_1= -v_3\wedge v,  \lambda  \beta_2= -v_3'\wedge v.
 \]
  But then 
   \begin{eqnarray}\nonumber
  (q_v+\lambda q_A)( \alpha)=v_2\otimes v_3\wedge v+v_2'\otimes v_3'\wedge v+\lambda v_2\otimes \beta_1+\lambda v_2'\otimes \beta_2\\
 \nonumber =v_2\otimes v_3\wedge v+v_2'\otimes v_3'\wedge v- v_2\otimes v_3\wedge v- v_2'\otimes v_3'\wedge v=0
   \end{eqnarray}
  so the implication  and the lemma follows. 
  \end{proof}
 \begin{rem} The intersection $S_A=\Sigma_{2,3}\cap Q_A$ is a del Pezzo surface of degree $2$.
 The plane $\PP(H^0({\mathcal I}_{\Sigma_{2,3}}(2)))\subset \PP(H^0(\mathcal{I}_{S_A}(2)))$ intersects the Kummer surface  $D_i^A\subset \PP(H^0(\mathcal{I}_{S_A}(2)))$ in a plane quartic curve.  One may show, that for general $A$, this curve is smooth.  
Considering the similar symmetric degeneracy locus of quadrics for a hyperplane section $S_A\cap H$ and a double hyperplane section $S_A\cap H\cap H'$, one may show that the corresponding plane quartics are a singular quartic and a double conic,  respectively.  
 \end{rem}
That the symmetric degeneracy locus ${\mathcal D}_1\subset \PP(H^0(\mathcal{I}_{S_A}(2)))$ has a component that is a Kummer surface can be seen considering conics on $S_A$.  The surface ${\mathcal D}_1$ is clearly a sextic, being the discriminant of a space of quadrics in $\PP^5$. Since the quadrics in the ideal of the Segre cubic scroll all have rank $4$, the plane $\PP(H^0({\mathcal I}_{\Sigma_{2,3}}(2)))$ is a component of multiplicity $2$ in ${\mathcal D}_1$, so the residual component is a quartic surface.  We show that $16$ pairs of conic curves on $S_A$  correspond to $16$ planes  in $\PP(H^0(\mathcal{I}_{S_A}(2)))$ that each contain  $6$ rank $4$-quadrics that contain $S_A$, but not $\Sigma_{2,3}$.  Furthermore there are $16$  rank $4$-quadrics on the quartic surface in ${\mathcal D}_1$ outside the plane $\PP(H^0({\mathcal I}_{\Sigma_{2,3}}(2)))$, so the quartic is a Kummer surface.

Let $\pi_1:S_A\to \PP^1$ and $\pi_2:S_A\to \PP^2$ be the two projections to the factors of $\Sigma_{2,3}$.  Then, for a general quadric $Q_A$ every line in the intersection $S_A =\Sigma_{2,3}\cap Q_A$ is contracted by the map $\pi_1$.  
\begin{prop}\label{32} Assume that $S_A$ is smooth and that every line in $S_A$ is contracted by $\pi_1$.  Then
\begin{enumerate}
\item $S_A$ contains $12$ lines, that form the components of $6$ singular conics. 
\item $S_A$ contains $32$ smooth conic sections  that are not fibers of $\pi_1$.  They form $16$ pairs that each intersect in a scheme of length $2$. 
\item $S_A$ contains $32$ pencils of twisted cubic curves, that are pairwise complementary in hyperplane sections. 
\end{enumerate}
\end{prop}
\begin{proof} The fibers of the projection $\pi_1:S_A\to \PP^1$ are plane conics, so $S_A$ is birational to a ruled surface.  Let $H$ be the class of a hyperplane section on $A$ and $F$ the class of a fiber, then the canonical divisor is, by adjunction on $\Sigma_{2,3}$, 
\[
K_{S_A}=-2H+F.
\]
So $K_{S_A}^2=2$ and $S_A$ is isomorphic  to a rational ruled surface blown up in $6$ points, and therefore has $6$ singular conics, i.e. $12$ lines that intersect in $6$ pairs and (1) follows. 
  
Consider next the projection $\pi_2:S_A\to \PP^2$.  It is $2:1$ and given by divisors in the class $H-F$.  The general curve in this class is an elliptic quartic curve which is mapped $2:1$ onto a line with $4$ branch points.  In particular, the branch curve in $\PP^2$ is a quartic curve with $28$ bitangent lines.  The preimage in $S_A$ of each of these lines is a pair of rational curves intersecting in $2$ points lying over the two branch points.
Now, every line $L$ in $S_A$ is mapped to a line by $\pi_2$, and $\pi_2^{-1}(\pi_2(L))$ is the union of $L$ and a twisted cubic curve $C_L$ with $C_L^2=-1$.
Since there are $12$ lines on $S_A$, there must be $16$ bitangents to the branch curve whose preimage in $S_A$ does not contain a line.  Since the preimages have degree $4$ on $S_A$, they must decompose into two smooth conics that intersect in a scheme of length $2$.  On the other hand, any conic that is not in a fiber of $\pi_1$ must be section of $\pi_1$ and is therefore mapped to a line by $\pi_2$, so (2) follows.  

Notice that each of these conic sections have self intersection $-1$ and intersect $15$ other conic sections among the $32$ in one point.

Consider any conic  section $C$ that is a section of $\pi_1$, and its complement $C'$ in the preimage of its image by $\pi_2$.  Then $C$ intersect $6$ lines in $S_A$, one from each singular fiber of $\pi_1$, while $C'$ intersect the remaining $6$.  Let $L$ be on of the lines intersecting $C$, then the divisor class $C+L$ contains a pencil of twisted cubic curves without basepoints on $S_A$.  If $L'$ is the line in $S_A$ that intersect $L$, then $C'+L'$ contains a  pencil of twisted cubic curves without basepoints and $C+L+C'+L'=H$.  Now, if $C''$ is a conic section in $S_A$ that do not intersect $C$, and $L''$ is a line that intersect $C''$ but neither of $C$ and $L$, then $(C+L)\cdot (C''+L'')=0$ and the two divisor classes $C+L$ and $C''+L''$ coincide.   Since  $(C'+L')\cdot (C''+L'')=3$,  we also have $C'\cdot C''=C'\cdot L''=C''\cdot L'=1$.
Let $L''$ be one of the $5$ lines in $S_A$ besides $L'$ that do not intersect $C$, then $C'\cup L'\cup L''$ spans a hyperplane, so the divisor class $H-C'-L'-L''$ contains a unique curve $C''$, a conic section that must be a section of $\pi_1$.  We may conclude that that in the pencil $|C+L|$ of twisted cubic curves there are $6$ singular fibers.  We conclude that each conic section $C$ that is a section of $\pi_1$ is a component of a fiber in $6$ pencils of twisted cubic curves, and that each such pencil has $6$ singular fibers.  Adding up we find $16$ pairs of base point free pencils of twisted cubic curves on $S_A$ and (3) follows.
\end{proof}
Notice that the linear span of each twisted cubic curve is contained in unique quadric that contains  $S_A$, a quadric of rank at most $4$ that does not belong to the ideal of $\Sigma_{2,3}$.  A hyperplane section of this quadric that contains the twisted cubic, will contain a twisted cubic of the complementary pencil, so the quadric must have rank $4$. On the other hand any rank $4$ quadric in the ideal of $S_A$ that does not contain $\Sigma_{2,3}$, will define on $S_A$ two base point free pencils of twisted cubic curves. We may therefore conclude:
\begin{cor}  In the ideal of $S_A$ there are exactly $16$ quadrics of rank $4$ that do not contain  $\Sigma_{2,3}$. Each of them define a pair of base point free pencils of twisted cubic curves on $S_A$.  Furthermore, let $C$ and $C'$ be a pair of conics in $S_A$ that intersect in a scheme of length $2$ and let $P$ and $P'$ be the planes spanned by these conics. Then the net of quadrics that contain $S_A$ and $P$ contains also $P'$, and the net contains exactly $6$ rank $4$-quadrics that do not contain  $\Sigma_{2,3}$.  
\end{cor}
\begin{proof} It remains only to remark that each quadric in the net that contain $S_A$ and $P$ contain both $C'$ and the line of intersection $P\cap P'$, so also $P'$.
\end{proof}

The dual surface $K^{\vee}$ to a Kummer quartic surface $K$ is also a Kummer quartic, with each plane tangent along a conic through $6$ nodes on $K$ corresponding to a node on $K^{\vee}$, so we conclude:
\begin{cor}  Let $D_1^A\subset \PP(H^0(\mathcal{I}_{S_A}(2)))$ be the Kummer surface, such that ${\mathcal D}_1= D_1^A\cup \PP(H^0({\mathcal I}_{\Sigma_{2,3}}(2))$.  Then the dual Kummer surface 
\[
(D_1^A)^{\vee}\subset \PP(H^0({\mathcal I}_{S_A}(2))^\vee)
\]
is singular in each point $[H^0({\mathcal I}_{S_A\cup \langle C\rangle}(2))]\in  \PP(H^0({\mathcal I}_{S_A}(2))^\vee)$, where $C\subset S_A$ is any of the $32$ conics whose spanning plane $\langle C\rangle$ is not contained in $\Sigma_{2,3}$. These conics occur in pairs that define the same point, thus accounting for the $16$ nodes of $(D_1^A)^{\vee}$.
 \end{cor}

\subsection{Kummer surfaces from a Hilbert scheme of conics}\label{hilbert-kummer}
We relate the Lagrangian and symmetric descriptions of Kummer surfaces to the Hilbert scheme of conics in a certain Fano threefold. 
 
 First we note a general lemma that identifies the discriminant locus of a family of quadrics with base locus a quadric section of a cone with the discriminant of the family of quadrics defining the branch locus of the induced double cover.
\begin{lemma}\label{discr} Let $X\subset \mathbb{P}^n$ be a manifold defined by quadrics and let $CX\subset  \mathbb{P}^{n+1}$ be a cone over $X$ with vertex $p\in \mathbb{P}^{n+1}$. Let $Q$ be a general quadric form  in $\mathbb{P}^{n+1}$.  Let $Y_Q=CX\cap \{Q=0\}$ and let $Y_r\subset X$ be the branch locus of the $2:1$ map induced by the projection from $p$ of $Y_Q$ onto $X$. Let 
$D_{CX}\subset \mathbb{P}(H^0(\mathbb{P}^{n+1},I_{CX}(2)))$ and $D_{Y_Q}\subset \mathbb{P}(H^0(\mathbb{P}^{n+1},I_{Y_Q}(2)))$  be the discriminants. The projective  space $\mathbb{P}(H^0(\mathbb{P}^{n+1}, I_{CX}(2)))$ is a hyperplane in $\mathbb{P}(H^0(\mathbb{P}^{n+1},I_{Y_Q}(2)))$, so we consider the inclusions
\[
D_{CX}\subset D_{Y_Q}\subset\mathbb{P}(H^0(\mathbb{P}^{n+1},I_{Y_Q}(2))).
\]
  Similarly, we consider the inclusions in  $ \mathbb{P}(H^0(\mathbb{P}^{n},I_{X}(2)))$ and  
  $\mathbb{P}(H^0(\mathbb{P}^{n},I_{Y_r}(2)))$
  \[
  D_X\subset D_{Y_r} \subset \mathbb{P}(H^0(\mathbb{P}^{n},I_{Y_r}(2))).
  \]
 Then there exists a linear isomorphism $\mathbb{P}(H^0(\mathbb{P}^{n},I_{Y_r}(2)))\to \mathbb{P}(H^0(\mathbb{P}^{n+1},I_{Y_Q}(2)))$ mapping $D_{Y_r} \setminus D_X$ isomorphically to $D_{Y_Q}\setminus D_{CX}$.
\end{lemma}
\begin{proof} Observe that in an appropriate choice of coordinates in $\mathbb{P}^{n+1}$ we have 
$$Q(z,x_0,\dots, x_n)=z^2-Q'(x_0,\dots, x_n)$$ 
and $p$ is the point $(0,\dots, 0, 1)$. It is the clear that in this setup $Y_r$ is defined in $\mathbb{P}^{n}$ with coordinates $x_0,\dots x_n$ as $X\cap\{(x_0:\dots :x_n)| Q'(x_0: \dots :x_n)=0\}$. Note that $H^0(I_{CX}(2))=H^0(I_{X}(2))$. Consider the map: 
$$\phi: H^0(\mathbb{P}^{n},I_{Y_r}(2)) \to H^0(\mathbb{P}^{n+1},I_{Y_Q}(2))$$
such that $\phi |_{H^0(I_{X}(2))}=id$ and $\phi(Q')=Q$. Clearly $\phi$ is an isomorphism that doesn't change the corank of the quadrics that do not belong to $I_{X}(2)$, while it increases the corank by one for each quadric in $I_{X}(2)$.  The complement $D_{Y_r} \setminus D_X$ is therefore isomorphic to $D_{Y_Q}\setminus D_{CX}$.
\end{proof}
Consider the $6$-space  $\PP(\C\oplus (V_2\otimes  V_3)) (=\PP^6)$, a general quadric hypersurface $Q_A$ in this space and the $3$-fold obtained as the intersection
 \[
 T_A=C(\PP(V_2)\times \PP(V_3))\cap Q_A\subset \PP(\C\oplus (V_2\otimes  V_3)).
 \]
 Denote by $p$ the vertex of $C(\PP(V_2)\times \PP(V_3))$, and let $H_{A,p}$ be the polar of $p$ with respect to the quadric $Q_A$, and let $Q_{A,p}=Q_A\cap H_{A,p}$ and $S_{A}=T_A\cap H_{A,p}$.   
Following Lemma \ref{discr},  the restriction map $H^0(\PP(\C\oplus (V_2\otimes  V_3),{\mathcal I}_{T_A}(2))\to H^0(H_{A,p},{\mathcal I}_{S_{A}}(2))$  is an isomorphism not just between the vector spaces, but also between the components of the discriminants
residual to the planes $\PP(H^0({\mathcal I}_{C(\PP(V_2)\times \PP(V_3))}(2)))$ and $\PP(H^0({\mathcal I}_{\Sigma_{2,3}}(2)))$, respectively.
The discriminant in $\PP(H^0({\mathcal I}_{T_A}(2)))$ is the union of the plane $\PP(H^0({\mathcal I}_{C(\PP(V_2)\times \PP(V_3))}(2)))$ and a surface that we therefore may identify with the Kummer surface $D_1^A$. Dual to $D_1^A$ is the Kummer surface $(D_1^A)^\vee\subset \PP(H^0({\mathcal I}_{T_A}(2))^{\vee}$.

 The $3$-fold $T_A$ has natural projections, $\pi_1:T_A\to \PP(V_2)$ and $\pi_2:T_A\to \PP(V_3)$.  A conic in $T_A$ that is mapped birationally to $\PP(V_2)$ and birationally onto a line in $\PP(V_3)$ is called a $(1,1)$-conic.
We denote by  $F(T_A)$ the Hilbert scheme of $(1,1)$-conics in $T_A$.  
\begin{prop}\label{1.10} $F(T_A)$ admits a morphism  
\[
\psi_{Q_A}: F(T_A)\to (D_1^A)^\vee\subset \PP(H^0({\mathcal I}_{T_A}(2))^{\vee}
\]
whose general fiber is a pair of $\PP^1$'s.
\end{prop}
\begin{proof}
The proof requires several lemmas.  First we define $\psi_{Q_A}$. 
For any $(1,1)$-conic $C\subset T_A$  we let $P_C$ be the plane spanned by $C$.
Then the subspace
$$H_C:= H^0({\mathcal I}_{T_A\cup P_C}(2))\subset H^0({\mathcal I}_{T_A}(2))
$$
has codimension one, and hence defines a point in  
\[
[H_C]\in
\PP(H^0({\mathcal I}_{T_A}(2))^{\vee}.
\]

We shall show that $(D_1^A)^\vee$ is the image of the map 
\[
\psi_{Q_A}: F(T_A)\to \PP(H^0({\mathcal I}_{T_A}(2))^{\vee},\qquad [C]\mapsto [H_C].
\]
First, however, we show that the general fiber of $\psi_{Q_A}$ is a pair of $\PP^1$'s.
\begin{lemma}\label{phi_C} Assume that $Q_A$ is general, so that $T_A$ is smooth. Let $[C]\in F(T_A)$, then the subscheme defined by the net of quadrics $H_{C}$ is a complete intersection, the union of $T_A$ and a quadric threefold $Q_{C}$ of rank at most $4$. 
For general $C$, the quadric  $Q_{C}$ has rank $4$ with singular point $p_C\notin T_A$, and the intersection $Q_{C}\cap T_A$ is a Del Pezzo quartic surface inside $T_A$.  The two pencils of planes in $Q_{C}$, define two pencils of $(1,1)$-conics  on $T_A$.
\end{lemma}
\begin{proof}
We first show that the quadrics in $H_{C}$ define a complete intersection.
Note that since $C$ is a $(1,1)$-conic, the plane $P_C$ is not contained in the cone $C(\PP(V_2)\times \PP(V_3))$, so the net of quadrics $H_{C}$ cannot contain the cone.    Therefore, the net of quadrics $H_{C}$ contains a pencil of quadrics that contain this cone.  The base locus of this pencil is the union of the cone and a $\PP^4_C$ that intersects the cone in a quadric $3$-fold $Q_{CC}$. If the net of quadrics $H_C$  contains the $\PP^4_C$, then $Q_{CC}$ is a component of $T_A$, against the genericity of $T_A$. Therefore every component in the base locus of $H_C$ has codimension $3$ and $H_{C}$ defines a complete intersection.

This base locus is therefore the union of $T_A$ and a quadric $3$-fold $Q_{C}$ in $\PP^4_C$.   Since $Q_{C}$ contains the plane $P_C$, it has rank at most $4$, with equality for  general $C$.  
The intersection $T_A\cap Q_{C}= Q_{CC}\cap Q_{C}$ is a Del Pezzo surface, which is smooth for a general $C$. In particular, the singular point $p_C$ of the quadric $Q_{C}$ cannot lie on this surface.   The two pencils of planes in $Q_{C}$, intersect $T_A$ in two pencils of conics, both of type $(1,1)$.  The fiber of  the map $\psi_{Q}^{-1}(H_{C})$ is therefore two disjoint $\PP^1$'s.  
\end{proof} 
\begin{cor} The Hilbert scheme of $(1,1)$-conics $F(T_A)$ is a threefold.
\end{cor}
\begin{proof} The general net of quadrics $H\subset H^0({\mathcal I}_{T_A}(2))$ defines a reducible complete intersection $T_A\cup Q$, where $Q$ is a quadric threefold. The quadric $Q$ is singular for a codimension one, i.e. $2$-dimensional family of nets $H$, in which case the pencil of planes in $Q$ intersect $T_A$ in $(1,1)$-conics.
\end{proof}
To identify the image of $\psi_{Q_A}$ with the  Kummer surface $(D_1^A)^\vee$, 
we show that the net of quadrics $H_C\subset  H^0({\mathcal I}_{T_A}(2)$  defines a plane $\PP(H_C)$ that is tangent to the discriminant $D_1^A$, so that the point $[H_C]\in (D_1^A)^\vee$.

First we show that when $C$ is a $(1,1)$-conic on $T_A$, then the net of quadrics $H_C$ contains a quadric $Q_c$ that is singular in the base locus of $H_C$.
\begin{lemma}\label{Q_c}  Let $[C]\in F(T_A)$, and let $Q_{C}$ be the quadric $3$-fold of rank at most $4$ in $\PP^4_C$, such that the base locus of $H_C$ is $T_A\cup Q_C$.   Let $p_C\in Q_{C}$ be the singular point.
Then there is at least one quadric $Q_c\subset \PP(\C\oplus V_2\otimes V_3)$ that belongs to  $H_{C}$  and is singular at $p_C$.
\end{lemma}
\begin{proof}
The net of quadrics $H_{C}$ defines a complete intersection 3-fold $T_A\cup Q_{C}$ of degree 8, and  $Q_{C}\subset \PP^4_C$.
There is a pencil of hyperplanes in $\PP(\C\oplus V_2\otimes V_3)$ that contain $\PP^4_C$. Every quadric in $H_{C}$ contains $p_C\in Q_{C}$, and has a tangent space at $p_C$  that contains $\PP^4_C$, so one of these quadrics, say $Q_c$ is singular at the point $p_C$.  
\end{proof}

The next lemma implies that the plane $\PP(H_C)\subset  \PP(H^0({\mathcal I}_{T_A}(2))$ is tangent to the discriminant surface $D_1^A$.
\begin{lemma}\label{singdisc} Let $W$ be a linear space of quadrics in a projective space  $P$ and let $Z\subset P$  be the base locus of the quadrics in $W$. Let $D\subset W$ be the discriminant.  If $[Q]\in W$ is a singular quadric with singular point at $p\in Z$, then the discriminant $D$ is singular at $[Q]$.  
\end{lemma}
\begin{proof}The tangent space to $D$ in $W$ at a quadric $[Q]$ that is singular at $p\in P$ is the hyperplane in $W$ of quadrics that vanish at $p$.  So if $p$ is in the base locus $Z$, then the hypersurface $D$ is singular at $[Q]$.
\end{proof}

Let $C\subset T_A$ be a general $(1,1)$-conic, let  $H_C$ be the net of quadrics vanishing on $T_A\cup P_C$, and let $T_A\cup Q_C$ be the base locus of $H_C$.
Let $p_C\in \PP(\C\oplus V_2\otimes V_3)$ be the singular point in the quadric $3$-fold $Q_C$ of rank $4$. Then, by Lemma \ref{phi_C}, $p_C\notin T_A$ and, by Lemma \ref{Q_c}, $p_C$ is a singular point of a quadric $Q_c$ in $H_C$.  Therefore, by Lemma \ref{singdisc}, $\PP(H_{C})\cap D_1^A$ is singular at $[Q_c]$, so $\PP(H_{C})$ is the tangent plane to $D_1^A$ at $[Q_c]$.

In particular $\psi_{Q_A}$ maps to 
 $(D_1^A)^\vee$.  Since $F(T_A)$ is a threefold and the fibers are curves, the map is onto. 

\end{proof}


\subsection{From the Hilbert scheme of conics to a Lagrangian degeneracy locus.}
Finally we relate the base of the fibration on the Hilbert scheme $F(T_A)$ directly to the Lagrangian degeneracy locus defined in \ref{lagrangiankummer}.
Let us consider the space 
 \[
 T_A=C(\PP(V_2)\times \PP(V_3))\cap Q_A\subset \PP(\C\oplus (V_2\otimes  V_3)).
 \]
 Choose a coordinate system in $\C \oplus (V_2\otimes  V_3)$ in such a way that $Q_A(z,x)=z^2-Q'_A(x)$, i.e. such that $z=0$ is the hyperplane polar to the vertex of the cone with respect to the quadric $Q_A$. Note that we then have $T_A\cap \{z=0\}= S_A$. The quadric $Q'_A$ corresponds to a symmetric map $q'_A: V_2\otimes V_3 \to (V_2\otimes V_3)^{\vee}$. Let now $V_4:=\C v_0 \oplus V_3$. Thus $\wedge^2 V_4= v_0\wedge V_3\oplus \wedge^2 V_3$ and:
 $$V_2\otimes \wedge^2 V_4= (V_2\otimes v_0\wedge V_3)\oplus (V_2\otimes \wedge^2 V_3).$$
We shall from now on interpret 
 $V_2\otimes \wedge^2 V_4$ as a subspace in $\wedge^3(V_2\oplus V_4)$. Then,
 up to choices of volume forms $vol_2$ and $vol_4$ in $V_2$ and $V_4$ respectively, we have a natural skew-symmetric form $\eta_{2,4}$ on $V_2\otimes \wedge^2 V_4$ induced by the wedge product. The decomposition 
 $V_2\otimes \wedge^2 V_4= (V_2\otimes v_0\wedge V_3)\oplus (V_2\otimes \wedge^2 V_3)$ is then a decomposition into a sum of two Lagrangian spaces with respect to $\eta_{2,4}$. Furthermore the graph  $A\subset (V_2\otimes V_3)\oplus (V_2\otimes \wedge^2 V_3)$ of $q'_A$ is also Lagrangian. 
Hence to $A$ we can associate a Kummer surface:

$$\hat{D}_1^A:=\{[U]\in \PP(V_4^{\vee}) | \ \dim (A\cap (V_2\otimes \wedge^2 U)) \geq 1\},$$
singular in 
$$\hat{D}_2^A:=\{[U]\in \PP(V_4^{\vee}) | \ \dim (A\cap (V_2\otimes \wedge^2 U)) \geq 2\}.$$

On the other hand the system of quadrics containing $T_A$ is naturally isomorphic to $V_4=\C \oplus V_3$ via 
$$V_4=\C \oplus V_3 \ni (z,v)\mapsto (z\cdot Q_A+ Q_v)\in H^0({\mathcal I}_{T_A}(2)),$$
with $Q_v$ defined by $Q_v(z,x)=x\wedge x\wedge v$ where $(z,x)\in \C\oplus(V_2\oplus V_3)$.

We, consider the map $\psi_{Q_A}: F(T_A)\to \PP(H^0({\mathcal I}_{T_A}(2))^{\vee})=\PP(V_4^{\vee})$ associating to a conic $C$ the system $H_C$ of quadrics vanishing on $T_A$ and the plane $P_C$ spanned by $C$ and prove:
 
\begin{prop}\label{hilb lagr} The map  $\psi_{Q_A}$ factors as $\rho_{Q_A}\circ \phi_{Q_A}$, where $\phi_{Q_A}:F(T_A)\to X_A $ is a $\PP^1$-fibration and
$\rho_{Q_A}:X_A\to \hat{D}_1^A$ is $2:1$ onto its image.  Furthermore $X_A$ is an abelian surface, and $\rho_{Q_A}$ is a double cover of its Kummer surface.
\end{prop}

In Proposition \ref{1.10} we showed that $(D_1^A)^{\vee}$ is the image of $\psi_{Q_A}$.  We shall now see that the image of $\psi_{Q_A}$ is in fact also described as $\hat{D}^1_A$.
\begin{lem}  For any $(1,1)$-conic $C\subset T_A$ i.e. $[C]\in F(T_A)$ we have $\psi_{Q_A}([C])\in \hat{D}_1^A$, furthermore if $C\subset S_A=T_A\cap \{z=0\}$ then $\psi_{Q_A}([C])\in \hat{D}_2^A$.
\end{lem}
\begin{proof} Fix the notation above. We start by describing the map $\psi_{Q_A}$ in coordinates. 
 Consider three general points $(z_1,\beta_1), (z_2,\beta_2), (z_3,\beta_3) \in C\subset T_A \subset C(\PP(V_2)\times \PP(V_3))$.
 By assumption, $\beta_i \in V_2\otimes V_3$ are decomposable tensors that can be written as $\beta_1=u_1\otimes (v_0\wedge v_1),  \beta_2=u_2\otimes (v_0\wedge v_2), \beta_3=(u_1+u_2)\otimes (v_0\wedge (v_1+v_2))$ (recall that we interpret elements of $V_3$ as two-forms $v_0\wedge *$) for appropriate choice of basis $(u_1,u_2)$ of $V_2$ and $(v_1,v_2,v_3)$ of $V_3$ satisfying $vol_4(v_0\wedge v_1\wedge v_2\wedge v_3)=1$, $vol_2(u_1\wedge u_2)=1$ . We keep this basis until the end of the proof. Clearly the component of $\psi_{Q_A}(C)\in \PP(V_4^{\vee})=\PP(\C\oplus \wedge^2 V_3)$ corresponding to the part $\wedge^2 V_3$ is then $v_1\wedge v_2$. We need to determine the remaining part of $\psi_{Q_A}(C)$. Let $\alpha_i=q'_A(\beta_i)\in V_2\otimes \wedge^2 V_3$ which is equivalent to  $\alpha_i+\beta_i \in A\subset (V_2\otimes V_3)\oplus (V_2\otimes \wedge^2 V_3)$ and implies also $Q'_A(\beta_i)=\alpha_i\wedge \beta_i$.
 Since $A$ is Lagrangian we have $(\alpha_i+\beta_i)\wedge(\alpha_j+\beta_j)=0$ for all $i,j$ which implies  
 $\alpha_i\wedge\beta_j=\alpha_j\wedge\beta_i=:c_{i,j}$ for $i\neq j$. 
Now 
$$Q_A(\lambda_1 (z_i,\beta_i)+\lambda_2 (z_j,\beta_j))=(z_i+\lambda z_j)^2- Q'_A(\lambda_1\beta_i+ \lambda_2 \beta_j)=$$
$$=(\lambda_1 z_i+\lambda_2 z_j)^2- (\lambda_1\alpha_i+ \lambda_2 \alpha_j)\wedge(\lambda_1\beta_i+ \lambda_2 \beta_j)$$
$$=\lambda_1^2 Q_A((z_i,\beta_i))+\lambda_2^2 Q_A((z_j,\beta_j))+2 \lambda_1\lambda_2 (z_iz_j-c_{i,j}).$$
But $Q_A((z_i,\beta_i))=0$, since $(z_i,\beta_i)\in C\subset T_A$. So we deduce that
$$Q_A(\lambda_1 (z_i,\beta_i)+\lambda_2 (z_j,\beta_j))= 2\lambda_1\lambda_2 (z_iz_j-c_{i,j})$$
Now 
$$(t_0 Q_A+t_1 Q_{(v_1\wedge v_2)^*}) (\lambda_1 (z_i,\beta_i)+\lambda_2 (z_j,\beta_j))=2t_0 \lambda_1\lambda_2 (z_iz_j-c_{i,j})+ 2t_1\lambda_1\lambda_2$$
it follows that the $\psi_{Q_A}(C)=[(z_iz_j-c_{i,j},  v_1\wedge v_2) ]\in \PP(\C\oplus \wedge^2 V_3))$ which means, in particular, that:
$$z_1z_2-c_{1,2}=z_1z_3-c_{1,3}=z_2z_3-c_{2,3}=:c_{C}.$$
If now $U_C=<[(c_C,  v_1\wedge v_2) ]>^\perp\subset \C\oplus V_3$ we have 
$$\wedge^2 U_C=\{\gamma\in (v_0\wedge V_3)\oplus \wedge^2 V_3 |  \gamma \wedge v_1\wedge v_2 =\gamma \wedge v_1\wedge (c_C v_3+v_0)=\gamma \wedge v_2\wedge (c_C v_3+v_0)=0 \}$$
 We deduce 
 $$V_2\otimes \wedge^2 U_C= \{\omega\in V_2 \otimes \wedge^2 V_4 |  \omega \wedge v_1\wedge v_2 =\omega \wedge v_1\wedge (c_C v_3+v_0)=\omega \wedge v_2\wedge (c_C v_3+v_0)=0 \}$$
 
We shall prove that $A\cap (V_2\otimes \wedge^2 U_C)\neq 0$. We know $\sum_{i=1}^3 \lambda_i (\alpha_i+\beta_i)\in A$ for $\lambda_i\in \C$. It is therefore enough to prove that the following system of equations has a nonzero solution $(\lambda_1,\lambda_2,\lambda_3)$:
\begin{equation*}
\begin{split}
E_1(\lambda_1,\lambda_2,\lambda_3):=(\sum_{i=1}^3 \lambda_i (\alpha_i+\beta_i))\wedge v_1\wedge v_2=0\\
E_2(\lambda_1,\lambda_2,\lambda_3):=(\sum_{i=1}^3 \lambda_i (\alpha_i+\beta_i))\wedge v_1\wedge (v_0+c_{C} v_3)=0\\
E_3(\lambda_1,\lambda_2,\lambda_3):=(\sum_{i=1}^3 \lambda_i (\alpha_i+\beta_i))\wedge v_2\wedge (v_0+c_{C} v_3)=0\\
\end{split}
\end{equation*}
Observe now that $E_1(\lambda_1,\lambda_2,\lambda_3)\equiv 0$  since both $\alpha_i\wedge v_1\wedge v_2=0$ and $\beta_i\wedge v_1\wedge v_2=0$.
Furthermore, we have:
$$E_2(\lambda_1,\lambda_2,\lambda_3)\wedge v_1=E_2(\lambda_1,\lambda_2,\lambda_3)\wedge v_2=E_2(\lambda_1,\lambda_2,\lambda_3)\wedge v_3=E_2(\lambda_1,\lambda_2,\lambda_3)\wedge v_0=0,$$
as well as 
$$E_3(\lambda_1,\lambda_2,\lambda_3)\wedge v_1=E_3(\lambda_1,\lambda_2,\lambda_3)\wedge v_2=E_3(\lambda_1,\lambda_2,\lambda_3)\wedge v_3=E_3(\lambda_1,\lambda_2,\lambda_3)\wedge v_0=0.$$
Finally, the three equations
$$E_2(\lambda_1,\lambda_2,\lambda_3)\wedge u_1= z_1^2 \lambda_1 +z_1z_2 \lambda_2+z_1z_3 \lambda_3=0,$$
$$E_3(\lambda_1,\lambda_2,\lambda_3)\wedge u_2= z_1z_2 \lambda_1 +z_2^2 \lambda_2+z_2z_3 \lambda_3=0,$$
$$(E_2(\lambda_1,\lambda_2,\lambda_3)+E_3(\lambda_1,\lambda_2,\lambda_3))\wedge (u_1+u_2)= z_1z_3 \lambda_1 +z_2 z_3 \lambda_2+z_3^2 \lambda_3=0,$$
are proportional, so the above equations reduce to the following linear equations in the $\lambda_i$:
$E_2(\lambda_1,\lambda_2,\lambda_3)\wedge u_2=0$
and one of the above proportional equations.
It follows that the system is a rank 2 system of equations, it hence admits a nontrivial solution implying 
$\dim (A\cap \wedge^2 U_C)\geq 1$, which proves:

$$\psi_{Q_A}([C])\in \hat{D}_1^A.$$
If now $C$ is contained in the branch locus $S_A=T_A\cap {z=0}$, then the system of equations is of rank 1 since the three above proportional equations vanish, so $\psi_{Q_A}([C])\in \hat{D}_2^A$.
\end{proof}

We can now pass to the proof of Proposition \ref{hilb lagr}.
\begin{proof}[Proof of Proposition \ref{hilb lagr}] Note that there is a natural involution on $F(T_A)$ induced by the involution 
$(z,\beta)\mapsto(-z,\beta)$.  A conic is fixed by this involution if and only if it is contained in $\{z=0\}$. Note also that from the explicit formula it follows that the involution acts on the fibers of  $\psi_{Q_A}$. From Lemma \ref{phi_C} we know that such a fiber is either a disjoint union of two $\PP^1$s or a single $\PP^1$. On the other hand, we know that a nontrivial involution on $\PP^1$ has  two fixed points. By Proposition \ref{32}, there are exactly 16 pairs of $(1,1)$-conics on $S_A=T_A\cap\{z=0\}$.  Each pair intersect in two points, so we deduce that the involution exchanges the $\PP^1$s  in the general fiber and acts on the $\PP^1$s in 16 fibers whose images are the 16 singular points of the Kummer surface $\hat{D}_1^A$. Hence the Stein factorization of  $\psi_{Q_A}$ gives the desired decomposition. Moreover, $X_A\to \hat{D}_1^A$ is a double cover branched in the 16 singular points of $\hat{D}_1^A$ and is therefore an Abelian surface. 
\end{proof}

Note that, as observed above, combining Proposition \ref{hilb lagr} with Proposition \ref{1.10} we obtain.
\begin{cor} \label{projdualKummer} The Kummer surface $\hat{D}_1^A$ is projective dual to the Kummer surface $D_1^A$.
\end{cor}
\begin{rem} Note that Corollary \ref{projdualKummer} provides further analogies between our description of Kummer surfaces and that of EPW sextics. Indeed a choice of Lagrangian $A$ provides two constructions leading to birational and projectively dual varieties $\hat{D}_1^A$ and $D_1^A$ which are both Kummer surfaces. In the context of EPW sextics a choice of Lagrangian space also gives rise to two birational and projectively dual EPW sextics.  
\end{rem}

\section{First construction - singular EPW cubes}\label{section2}
In this section we present the first construction of the family $\mathcal{U}$.
Let us first discuss a natural context where the elements from $\mathcal{U}$ appear. We shall investigate IHS fourfolds deformation equivalent to the Hilbert scheme of two points on a $K3$ surface (of $K3^{[2]}$-type) which admit an antisymplectic involution $\iota$ (i.e. that changes the sign of the symplectic form). 

 Involutions of $K3$  surfaces were first studied from a lattice-theoretic point of view by Nikulin  \cite{Nikulin}. For higher dimensions a classification of invariant lattices of non-symplectic automorphisms of
prime order was given in \cite{BCS} and \cite{BCMS}. The problem of finding a geometric realization of non-symplectic automorphisms on IHS fourfolds was addressed in \cite{OW} and \cite{MongardiWandel}.

It follows from \cite[3.4, Theorem 2]{beau-invo} and \cite{OgradyEPW} that there exists exactly one irreducible 20-dimensional family of IHS fourfolds of $K3^{[2]}$ 
type  which admit anti-symplectic involutions. By \cite{OgradyEPW}, the invariant polarisation in this family has Beauville degree $q=2$ and the quotient of such an involution for a generic element is a special sextic hypersurface in $\PP^5$ called an EPW sextic. 
In \cite{OW} the authors classified all the possible invariant lattices $H^2(X,\mathbb{Z})^{\iota}$ of $19$ parameter families of IHS fourfolds of $K3^{[2]}$ type.
They found that any such lattice is hyperbolic and $2$-elementary. In \cite[thm.~2.3]{OW} they
distinguished five families of IHS fourfolds of $K3^{[2]}$-type with anti-symplectic involutions. 
In fact there are only four isomorphism classes of invariant sublattices $H^2(X,\Z)^{\iota}\subset H^2(X,\Z)$.  
They are $U$, $U(2)$, and $\langle 2\rangle \oplus \langle -2\rangle$ such that the generator $g$ with $q(g)=-2$ has divisibility either $2$ or $1$ (we call them the cases 1,2,3,4 respectively).
 Moreover, they found that there is a unique $19$-dimensional irreducible family that admits the invariant lattice from each of the cases $1,2,4$ respectively and two families in the case $3$. The families in the cases $1,3,4$ admit polarisations of Beauville degree $q=2$, it is not hard to see \cite[Rem.~5.7]{MongardiWandel} that they can be described as families of resolutions of special singular double EPW sextics. 

Our aim is to study the geometry of the missing family of IHS fourfolds with involutions  from the case $2$ above i.e. with invariant lattice $U(2)$. Note that each element of this family admits a natural polarization of Beauville degree $4$ and as proved in \cite[Proposition.~4]{Addington} the generic element of this family is not isomorphic to a moduli space of sheaves on a $K3$-surface. Note finally that the family with invariant lattice $U$ also admits a polarisation of degree $q=4$ that is invariant with respect to the non-symplectic involution. 

From \cite{OW} there is only one possible invariant lattice of rank two
$$H^2(X,\mathbb{Z})^{\iota}:=\{ x\in H^2(X,\mathbb{Z})| \iota^{\ast}(x)=x \}.$$
 that does not admit a polarization of Beauville degree $q=2$, namely:
$$U(2)=\begin{bmatrix} 0&2\\
2&0
\end{bmatrix}.$$
Let $X$ be an IHS fourfold with an involution and an invariant lattice $U(2)$.
 Then, by \cite{beau-invo}, the invariant lattice
has signature $(1,1)$ for some $n$. In particular $X$ is projective.
Let $h_1$ and $h_2$ be the generators of the lattice with $q(h_1)=q(h_2)=0$.
We are interested in the invariant polarization $h=h_1+h_2$ of Beauville degree $q(h_1+h_2)=4$.
If $(X,H)$ is a polarized IHS fourfold of type $($K3$)^{[2]}$ with $q(H)=4$, we infer
$H^4=3\cdot (4)^2=48$, and from the Riemann--Roch theorem we find $h^0(\mathcal{O}_X(H))=10$ \cite[Theorem~5.2]{Nieper}.
Thus, our polarization $h$ gives a map $f\colon X\to \PP^9$ that factors through the involution $\iota$.
Hence, we expect that $f$ is $2:1$ to a degree $24$ fourfold. Our aim is to describe the image of this map.
We shall first show that this image can be realized as a subset of a degenerated EPW cube and next prove that in fact $X$ is an element of $\mathcal{U}$. 

\subsection{Degenerate EPW cubes}\label{Section-Singular}
In this section we consider double EPW cubes constructed from a general Lagrangian subspace $A\in \Sigma$, in particular with  $\PP(A)\cap \mathrm{G}(3,V)=[U_1]$.
Let $\mathcal{T}$ be the Lagrangian subbundle $\mathcal{T}\subset \mathcal{O}_{\mathrm{G}(3,V)}\otimes \wedge^3 V$ whose fiber over $[U]\in \mathrm{G}(3,V)$  is $\mathcal{T}_{[U]}=T_U=\wedge^2 U\wedge V$.   The degeneracy locus $D_2^A=\{[U]\in \mathrm{G}(3,V)| \dim (A\cap T_U)\geq 2\}$ is called an EPW cube.   
Our $19$-dimensional family of IHS fourfolds will be constructed from the subvariety $D_2^A\cap \PP(T_{U_1})$, when $\PP(A)\cap \mathrm{G}(3,V)=[U_1]$.  

The following description of a projective tangent space $\PP(T_U)$ to $\mathrm{G}(3,V)$ at $[U]$ is classical \cite{Donagi}.
\begin{lem}\label{lem-C_U} Let $[U]\in \mathrm{G}(3,V)$ and $\PP(T_{U})=\PP(\wedge^2{U}\wedge V)\subset \PP(\wedge^3V)$ be the embedded projective tangent space to $\mathrm{G}(3,V)\subset \PP(\wedge^3 V)$ in $[U]$.
Then the intersection
$$C_{U}:=\PP(T_{U})\cap \mathrm{G}(3,V)$$ 
 is a cone in the $9$-dimensional linear space $\PP(T_{U})$ with vertex $[U]$ over the Segre embedding of $\PP(\wedge^2U)\times \PP(V/U)$.
\end{lem}
\begin{proof} See \cite[Lemma 3.5]{Donagi}.  The tangent space $\PP(T_U)$ is spanned by the spaces $U'$ that intersect $U$ in codimension $1$.  These spaces are naturally parameterized by pairs $(M_2,N_1)$, where $M_2\subset U$ is $2$-dimensional, $N_1\subset V/U$ is $1$-dimensional and $U'\cap U=M_2$, $U'/M_2=N_1$.  
\end{proof}
\medskip
Note that for each $[U]\in C_{U_1}$ we have $[U_1]\in \mathbb{P}(T_U)\cap\PP(A)$. It follows that $C_{U_1}\subset D_1^A$.  Observe, that since $A$ is Lagrangian and $\wedge^3 U_1\subset A$
then $A\subset (\wedge^3 U_1)^{\perp}$. Similarly $T_U\subset (\wedge^3 U_1)^{\perp}$ for each $[U]\in C_{U_1}$. There is moreover an exact sequence:

\begin{equation}\label{egbart}\textstyle
0\to \mathcal{O}_{C_{U_1}} \to \mathcal{T} |_{C_{U_1}} \to \bar {\mathcal{T}} \to 0, 
\end{equation}
with $\bar{\mathcal{T}}$ a subbundle of the trivial bundle $\mathcal{O}_{C_{U_1}}\otimes (\wedge^3 V)/(\wedge^3 U_1)$ with fibers $\bar{T}_U= T_U/(\wedge^3 U_1)$ over $[U]\in C_{U_1}$.
Consider the space $$(\wedge^3 U_1)^{\perp}/(\wedge^3 U_1) \subset (\wedge^3 V)/(\wedge^3 U_1)$$ equipped with the symplectic form $\bar{\eta}$ induced by $\eta$. Clearly both $\bar{A}=A/(\wedge^3 U_1)$ and the fibers $\bar{T}_U= T_U/(\wedge^3 U_1)$ of $\bar{\mathcal{T}}$ are contained in $ (\wedge^3 U_1)^{\perp}/(\wedge^3 U_1)$  and are Lagrangian with respect to the symplectic form $\bar{\eta}$.  The natural map
\[
\iota: C_{U_1}\to \mathrm{LG}_{\bar{\eta}}(9,(\wedge^3 U_1)^{\perp}/(\wedge^3 U_1));\quad [U]\mapsto [\bar{T}_U]
\]
is an embedding, since it is the restriction of the embedding 
\[
\mathrm{G}(3,V)\to \mathrm{LG}_\eta(10,\wedge^3V), \;\; [U]\mapsto [T_U]
\]
 to $C_{U_1}$.
Denote the corresponding Lagrangian degeneracy loci by
$$D_k^{\bar{A}}=\{ [U]\in C_{U_1}| \dim (\bar{T}_U \cap \bar{A})\geq k\}.$$
These degeneracy loci are simply the restrictions to $\iota(C_{U_1})$ of the universal degeneracy loci $\mathbb{D}_k^{\bar{A}}$ on $\mathrm{LG}_{\bar{\eta}}(9,(\wedge^3 U_1)^{\perp}/(\wedge^3 U_1))$ \cite{PragaczRatajski}.

\begin{lem}\label{lem-ram} Let $[A]\in (\Sigma-(\Sigma_+\cup \Sigma[1])) \subset \mathrm{LG}_\eta(10,\wedge^3W)$ such that $[U_1]\in \mathbb{P}(A)\cap \mathrm{G}(3,V)$. Then $C_{U_1}\subset D_1^A$, and  $D_i^{\bar{A}}=C_{U_1}\cap D_{i+1}^A$, when $i=1,2$.  
Furthermore,
$D_1^{\bar{A}}$ is an intersection of $C_{U_1}\subset \PP^9$ with a quartic hypersurface $Q_A$,
and $D_2^{\bar{A}}$ is a surface of degree $72$ contained in the singular locus of $D_1^{\bar{A}}$.
\end{lem}
\medskip
\begin{proof} 
First, we simply note that $\bar{A}\cap \bar{T}_U =(A\cap T_U)/(\wedge^3 U_1)$, so we obtain 
$C_{U_1}\cap D_{i+1}^A=D_i^{\bar{A}}$.   
To compute invariants, recall that
the $\PP^9$-bundle $\PP(\mathcal{T})$ is the projective tangent bundle on $\mathrm{G}(3,V)$, 
so $\mathcal{T}^{\vee}$ fits into an exact sequence
\[
0\to\Omega_{\mathrm{G}(3,V)}(1)\to \mathcal{T}^{\vee} \to \mathcal{O}_{\mathrm{G}(3,V)}(1)\to 0.
\]
Therefore ${\mathcal T^{\vee} }$ has total Chern class 
\begin{align*}\label{}
c({\mathcal T^{\vee} })= \; &c(\mathcal{O}_{\mathrm{G}(3,V)}(1))/(\Omega_{\mathrm{G}(3,V)}(1))\\
= \; &1+4\sigma_1+8\sigma_1^2+(8\sigma_1^2+6\sigma_1\sigma_2-6\sigma_3)+(24\sigma_1^2\sigma_2-24\sigma_1\sigma_3)\\
&+(30\sigma_1\sigma_2^2-30\sigma_2\sigma_3)+(10\sigma_2^3+24\sigma_1\sigma_2\sigma_3-24\sigma_3^2)+
18\sigma_2^2\sigma_3+12\sigma_2\sigma_3^2+4\sigma_3^3 
\end{align*}
where $\sigma_i=c_i(Q_G)$ and $Q_G$ is the universal quotient bundle on $\mathrm{G}(3,V)$.
Furthermore, by the exact sequence \ref{egbart},  $c_i(\overline {\mathcal{T}}^{\vee} )\cap C_{U_1}=c_i({\mathcal{T^{\vee} }})\cap C_{U_1}$ for all $i$.  Applying the Pragacz Ratajski formulas \cite[Theorem 2.1]{PragaczRatajski} for the classes of the Lagrangian degeneracy loci $D_i^{\bar{A}}$ we get
\[
[D^{\bar{A}}_1]=c_1(\overline{{\mathcal T}}^{\vee})\cap [C_{U_1}] =c_1(\mathcal{T}^{\vee})\cap [C_{U_1}]=4\sigma_1\cap [C_{U_1}],
\]
so  $D_1^{\bar{A}}$ is an intersection of $C_{U_1}\subset \PP^9$ with a quartic hypersurface $Q_A$.
Furthermore
\begin{align*}
 [D^{\bar{A}}_2]&=(c_2c_1-2c_3)(\overline{{\mathcal T}}^{\vee})\cap [C_{U_1}]=(c_2c_1-2c_3)({\mathcal T}^{\vee})\cap [C_{U_1}] \\
 &=(c_2c_1-2c_3)({\mathcal T}^{\vee})\cap [C_{U_1}] = (16\sigma_1^3-12\sigma_1\sigma_2+12\sigma_3)\cap [C_{U_1}].
\end{align*}
 The class of $[C_{U_1}]$ in $\mathrm{G}(3,V)$ is $(\sigma_2^2-\sigma_1\sigma_3)\cap [\mathrm{G}(3,V)]$, 
 so 
 \[
 \deg D^{\bar{A}}_2=\int_{ [\mathrm{G}(3,V)]}\sigma_1^2\cdot (\sigma_2^2-\sigma_1\sigma_3)\cdot (16\sigma_1^3-12\sigma_1\sigma_2+12\sigma_3)=\int_{ [\mathrm{G}(3,V)]}36 \sigma_1^2\sigma_2^3\sigma_3=72.
 \]
 The last statement is a standard result on degeneracy loci.
\end{proof}
To proceed with the construction we need to know precisely the singular locus of the Lagrangian degeneracy locus $D_1^{\bar{A}}$.
  \begin{lem} \label{singularities of bar D1A} Let $[A]\in (\Sigma-(\Sigma_+\cup \Sigma[1])$  and let $[U_1]$ be the unique point in $\Theta(A)$. 
  Then the Lagrangian locus $C_{U_1}\cap D_2^A=D_1^{\bar{A}}$ is smooth outside $S_{\bar{A}}=D_2^{\bar{A}}=C_{U_1}\cap D_3^A$. 
  Moreover, the tangent cone to $D_1^{\bar{A}}$ in points of $S_{\bar{A}}$ is a cone over a smooth conic curve.
  \end{lem}
  \medskip
\begin{proof} The proof will be analogous to that  of \cite[Lemma 2.9]{EPWcubes}. 
 Let $[U_1]$ be the unique point in $\Theta(A)$. Observe that, by assumption, $[U_1]\notin D_1^{\bar{A}}$. 
 Fix $[U_0]\in C_{U_1}\cap D_1^{\bar{A}}$ and choose a $3$-space $U_\infty$ such that $U_{\infty}\cap U_1=0$ and $U_{\infty}\cap U_0=0$. 
 Let $$\mathfrak{U}=\{[U]\in \mathrm{G}(3,V)| U\cap U_\infty=0 \}.$$
 It is an open neighbourhood of $[U_0]$ in $\mathrm{G}(3,V)$.
 
 For $[U]\in \mathfrak{U}$ the Lagrangian space $T_U$ defines a symmetric linear map $T_{U_0}\to T_{U_0}^\vee$ that we denote by $q_U
 $ and a corresponding  quadratic form on $T_{U_0}$ that we denote by $Q_U
 $. We shall describe $Q_U$ in local coordinates.  Let $(u_1,u_2,u_3), (u_4,u_5,u_6)$ be a basis for $U_0$ resp. $U_{\infty}$.
 
 Observe that  for any $[U]\in \mathrm{G}(3,V)$,
  \[
T_U\cap T_{U_\infty}=0\leftrightarrow  U\cap {U_\infty}=0
 \] 
 and that any such subspace $U$ is the graph of a linear map $\beta_U:U_0\to U_{\infty}$.
 In particular, there is an isomorphism:
 \[
\rho: \mathfrak{U}\to Hom(U_0,U_{\infty}); \quad [U]\mapsto \beta_U
\]
whose inverse is the map
\[
 \alpha \mapsto [U_\alpha]:=[(u_1+\alpha(u_1))\wedge(u_2+\alpha(u_2))\wedge (u_3+\alpha(u_3))].
\]
In the given basis for $U_0$ and $U_\infty$ we let $B_U=(b_{i,j})_{i,j\in\{1\dots 3\}}$ be the matrix of the linear map $\beta_U$. In the dual basis, we let $(m_0, M)$, with $ M=(m_{i,j})_{i,j\in\{1\dots3\}}$, be the coordinates in
\[
T_{U_0}^{\vee}=(\wedge^3 U_0\oplus \wedge^2U_0\otimes U_\infty)^\vee=(\wedge^3 U_0\oplus Hom(U_0, U_\infty))^\vee
\]
  In these coordinates, the map 
  \[
  \iota: \mathfrak{U}\ni [U]\mapsto Q_U\in Sym^2 T_{U_0}^{\vee}
  \]
  is defined by
  \begin{equation}\label{eqquadrics}
  Q_U(m_0, M)=\sum_{i,j\in\{1\dots3\}} b_{i,j} M^{i,j} + m_0 \sum_{i,j\in\{1\dots3\}} B^{i,j}_U m_{i,j} + m_0^2 \det B_U, 
  \end{equation}
  where $M^{i,j}$, $B^{i,j}_U$ are the entries of the matrices adjoint to $M$ and $B_U$.
 To see this, write  the map $\wedge^3 U_0 \oplus \wedge^2 U_0\otimes U_{\infty} \to \wedge^3 U_{\infty} \oplus \wedge^2 U_{\infty}\otimes U_0  $
  whose graph is $ \wedge^3 U \oplus \wedge^2 U\otimes U_{\infty} $ in coordinates, where $U$ is the graph of the map $U_0\to U_{\infty}$ given by the matrix $B_U$.

  Let now $q_A$ be the symmetric map  $T_{U_0}\to T_{U_\infty}=T_{U_0}^{\vee}$ whose graph is $A$ and $Q_A$ the corresponding quadratic form.
  In this way
  $$D_l^A\cap\mathfrak{U}=\{[U]\in \mathfrak{U} |\dim T_U\cap A)\geq l\}=\{[U]\in \mathfrak{U} |\operatorname{rk} (q_U-q_A)\leq 10-l\},$$
  hence $D_l^A$ is locally defined by the vanishing of the $(11-l)\times (11-l)$ minors of the $10\times 10$ matrix with entries being polynomials in $b_{i,j}$.

We now consider the restriction of the map $\iota : [U] \mapsto Q_U$ to  $C_{U_1}\cap \mathfrak{U}$.  

The map $f: U_0\to U_{\infty}$, whose graph is $U_1$, has rank 1, since $[U_0]\in C_{U_1}\cap D_1^{\bar{A}}$ and $C_{U_1}\cap \mathfrak{U}=\{g\in Hom(U_0,U_\infty)| \rk (g-f)\leq 1\}$. 
After possible changes of basis for  $U_0$ and $U_{\infty}$, we may assume that $f\in Hom(U_0,U_\infty)$ is given by a matrix with one nonzero entry in the upper left corner.
The restriction of the map  $\iota$ is then given by
\begin{equation} \label{equationsqu}
Q_U(m_0,M)=\sum_{i,j\in\{1\dots3\}}  b_{i,j} M^{i,j}+ m_0 \sum_{i,j\in\{2,3\}} B_U^{i,j}m_{i,j}.
\end{equation}

We now observe that all quadrics $Q_U$ with $[U]\in C_{U_1}$ are singular in the point $[U_1]$ with coordinates $m_0=m_{1,1}\not=0$ and $m_{i,j}=0; (i,j)\not=(1,1) $. Passing to the quotient $T_{U_0}/\wedge^3 U_1$ and denoting the induced quadrics on the quotient space by $\bar{Q}_U$ and $\bar{Q}_{A}=Q_{\bar{A}}$ and the corresponding symmetric linear maps by $\bar{q}_U$ and $q_{\bar{A}}$ respectively. We have 
\begin{equation}\label{9x9quadrics}
\bar{Q}_U(M)=\sum b_{i,j} M^{i,j}.
\end{equation}

We can now follow the proof of \cite[Lemma 2.9]{EPWcubes} for the first degeneracy locus $D_1^{\bar{A}}$ around $[U_0]$. In $\mathfrak{U}$ the locus 
$$D_1^{\bar{A}}\cap \mathfrak{U}=\{[U]\in \mathfrak{U}\cap C_{U_1} | \dim (T_U/\wedge^3 U_1) \cap \bar{A})\geq l\} = \{[U]\in \mathfrak{U}\cap C_{U_1} |\operatorname{rk} (\bar{q}_U-q_{\bar{A}})\leq 9-l\},$$ i.e. $D_1^{\bar{A}}$
is defined by the determinant of a $9\times 9$ symmetric matrix with entries being regular functions on  $\mathfrak{U}\cap C_{U_1}$. We may assume that $q_{\bar{A}}$ is given by a diagonal matrix  with $0$'s and $1$'s on the diagonal, and let $K:=\ker q_{\bar{A}}=\bar{A}\cap T_{U_0}/\wedge^3 U_1$.  By \ref{equationsqu}, we know that the differential of  the map $\iota|_{C_{U_1}}$ in $[U_0]$ maps onto the linear system of quadrics generating the ideal of the image $\hat{C}$ of the projection of the cone $C_{U_0}\subset \mathbb{P}(T_{U_0})$ from the point $[U_1]$. In other words, the linear forms of the matrix of polynomials
$$\bar{\iota}: \mathfrak{U}\cap C_{U_1}  \ni [U]\mapsto \bar{Q}_U\in Sym^2 (T_{U_0}/\wedge^3 U_1)^{\vee}$$  
for a chosen coordinate chart of $C_{U_1}$ in $[U_0]$ define the linear system of quadrics containing $\hat{C}$. We then observe that if $\mathbb{P}(A)\cap \mathrm{G}(3,6)=[U_1]$ then $\mathbb{P}(A/\wedge^3U_1) \cap \hat{C}=\emptyset$ hence $K\cap  \hat{C}=0$ and remark that $\hat{C}$ satisfies the assertion of \cite[Lemma 2.8]{EPWcubes}. More precisely, we have:
  \begin{lem}\label{restriction of quadrics to K} If $P\subset \PP(T_{U_0}/\wedge^3 U_1)\setminus \hat{C}$ is a linear subspace of dimension at most $1$,  then the restriction map $\mathbf{r}_P: H^0(\PP(T_{U}), \mathcal{I}_{\hat{C}}(2)) \to H^0(P, \mathcal{O}_{P}(2))$ is surjective.
  \end{lem}
  \begin{proof} Note that $\hat {C}$ is defined in $\mathbb{P}^8= \PP(T_{U_0}/\wedge^3 U_1)$ by 5 quadrics obtained as $2\times 2$ minors  of a $3\times 3$ matrix of linear forms that do not involve the upper left entry. Since $\hat {C}$ is defined by quadrics the lemma is proven for $\dim P=0$. If $\dim P=1$ it is enough to observe that $\hat {C}$ can be seen as a cone over a section of the Grassmannian $\mathrm{G}(2,5)$ by two hyperplanes. Let $G=\mathrm{G}(2,5)\subset \PP^9$ and consider the rational map $\delta: \PP^9 \to \PP^4$ defined by the quadrics that generate the ideal of $G$. Observe that the closures of the fibers of $\delta$ are $\mathbb{P}^5$ spanned by 4-dimensional quadrics in $G$. It follows that the image $\delta(l)$ of any line $l$ , with $l\cap G=\emptyset$ is a smooth conic.  If now $CG\subset \PP^{10}$ is a cone over $G$, then the map defined by quadrics containing $CG$ factorizes as the  composition $\delta\circ \pi_p$ of the projection from the vertex $p$ of the cone $CG$ and $\delta$. It follows that $\delta(\pi_p (l))$ is a conic for any line $l\subset \PP^{10}$ such that $l\cap CG=\emptyset$. This means that the restriction map from the  system of quadrics containing $CG$ to quadrics on the line $l$ is surjective if $l\cap CG=\emptyset$. Now, since $\hat{C}$ appears as a section of $CG$ we conclude that the system of quadrics containing $\hat{C}$ contains the system of restrictions of quadrics containing $CG$. The latter restricts surjectively onto quadrics on the line $P$ since $P\cap CG=\emptyset$, which proves the lemma.\end{proof} 
Let us now denote the components of $\Phi:=\det (\bar{q}_U-q_{\bar{A}})$ of degree $i $ by $\Phi_i$. 
If now $[U_0]\in D_1^{\bar{A}}\setminus D_2^{\bar{A}}$ then $\dim K=1$, then $\Phi_0=0$ and $\Phi_1$ is the linear entry of $(\bar{q}_U- q_{\bar{A}})$ corresponding to the restrictions to $Sym^2 K$. It follows, by Lemma \ref{restriction of quadrics to K}, that  $\Phi_1\neq 0$ hence $D_1^{\bar{A}}$ is smooth in $[U_0]$. 
If now $[U_0]\in D_2^{\bar{A}}$, then $\dim K=2$ so $\Phi_0=\Phi_1=0$ and then  $\Phi_2$ is the determinant of the restriction of $\bar{q}_U-q_{\bar{A}}$ to $Sym^2 K$. Again, by Lemma \ref{restriction of quadrics to K}, we get that  $\Phi_2$ is a rank 3 quadric  which concludes the proof.
  
  \end{proof}

\begin{lem}\label{lem int} The variety $D_1^{\bar{A}}$ is integral.
\end{lem}
\begin{proof}  By Lemma \ref{singularities of bar D1A},  $D_1^{\bar{A}}$ is a divisor in $C_{U_1}$ that is smooth outside the codimension two locus $D_2^{\bar{A}}$; in particular it is reduced.
By Lemma \ref{lem-ram}, the locus $D_1^{\bar{A}}$ is the intersection of $C_{U_1}$ with a quartic hypersurface,  so if it was reducible, it would have singularities in codimension one which would contradict Lemma \ref{singularities of bar D1A}.
Therefore  $D_1^{\bar{A}}$ is integral.
\end{proof}

From Lemmas \ref{singularities of bar D1A} and \ref{lem int} we conclude that $D_1^{\bar{A}}$ is an irreducible $4$-fold  with quadratic singularities along the surface $D_2^{\bar A}$.   We proceed to construct a natural resolution of singularities.
For this define the incidences 
\[
 \tilde{D}_1^{\bar{A}}=\{([U],[\omega])\in C_{U_1}\times \mathrm{G}(1,\bar{A})| \quad \overline{T}_U \supset \langle \omega\rangle\},
 \]
 and 
 \[\textstyle 
 \tilde{\mathbb{D}}_1^{\bar{A}}=\{([L],[\omega])\in \mathrm{LG}_{\bar{\eta}}(9,(\wedge^3 U_1)^{\perp}/(\wedge^3 U_1))\times \mathrm{G}(1,\bar{A})| \quad L \supset \langle\omega\rangle\},
 \]
which fit in the following diagram:
\begin{center}
\begin{tikzcd}
 C_{U_1}              \rar{\iota}   &  \mathrm{LG}_{\bar{\eta}}(9,(\wedge^3 U_1)^{\perp}/(\wedge^3 U_1))\\
 D_1^{\bar{A}}   \rar[swap]{\iota|_{D_1^{\bar{A}}}} \arrow[Subseteq]{u}{}   & \mathbb{D}^{\bar{A}}_1 \arrow[Subseteq]{u}{}\\
\tilde {D}_1^{\bar{A}} \uar[swap]{\bar{\alpha}} \rar[swap]{\tilde{\iota}} & \tilde{\mathbb{D}}^{\bar{A}}_1 \uar[swap]{\phi}
\end{tikzcd}
\end{center}
 where $\bar{\alpha}$ and $\phi$ are the projections on the first factor.
 \begin{lem} The variety $\tilde {D}_1^{\bar{A}}$ as well as the exceptional divisor $E$ of $\bar{\alpha}$ are both smooth. In particular $\bar{\alpha}$ is a resolution of singularities of $ D_1^{\bar{A}}   $.
 \end{lem}
 \medskip
 \begin{proof} Once we have proved Lemma \ref{singularities of bar D1A}, and observed that $D_3^{\bar{A}}=\emptyset$ the proof is completely analogous to \cite[Lemma 3.3]{EPWcubes}.
 \end{proof}
 We can now perform the construction of a smooth double cover of $D_1^{\bar{A}}$ branched in $D_2^{\bar{A}}$. Note that the exceptional divisor in
 $\tilde {D}_1^{\bar{A}}$ is an even divisor.  To see this,  denote by $H$ a Pl\"ucker hyperplane section on  $\mathrm{LG}_{\bar{\eta}}(9,(\wedge^3 U_1)^{\perp}/(\wedge^3 U_1))$, denote by $h$  a Pl\"ucker hyperplane section on $\mathrm{G}(3,V)$ restricted to $C_{U_1}$, and denote by $R$ a Pl\"ucker hyperplane section on $\mathrm{G}(1,\bar{A})$.  Then, by  \cite[Lemma 2.4]{EPWcubes} and the fact that $\iota^*(H)=c_1({\mathcal T}^{\vee})\cap {C_{U_1}})=4h$, the divisor $E$ can be expressed as
 $$E=\iota^*(H-2R)=4h-2\iota^*(R).$$
 Hence, $E$ is divisible by $2$ and there exists a unique double cover $\tilde{f}: \tilde{X}_{\bar{A}}\to \tilde{D}_1^{\bar{A}}$ branched along $E$. Clearly the preimage of $\tilde{f}^{-1}(E)$ is contracted by a birational morphism $\psi$ defined by some multiple of the system $\tilde{f}^*\bar{\alpha}^* H$ on $\tilde{X}_{\bar{A}}$. The proof that the image $X_{\bar{A}}=\psi(\tilde{X}_{\bar{A}})$ 
 of this morphism is smooth is similar to the proof of \cite[Proposition 3.1]{EPWcubes}. It amounts to observing that the restriction of $\phi$ to the strict transform on $\tilde{X}_{\bar{A}}$ of a generic surface linear section of $D_1^{\bar{A}}$ 
  is the contraction of $(-1)$-curves on a smooth surface.  
  Denote by
\begin{equation}\label{eq}
p_X\colon X_{\bar{A}}\to D_1^{\bar{A}} 
\end{equation}
the induced double cover, ramified over $D_2^{\bar{A}} $ .

Let $[A]\in (\Sigma-(\Sigma_+\cup \Sigma[1]\cup \Gamma)$ and let
 $[U_1]=\PP(A)\cap \mathrm{G}(3,V)$.  In \cite[Section 3]{EPWcubes} a $6$-fold double cover $Y_A\to D_2^A$ ramified along $D_3^A$ is constructed over the second degeneracy locus $D_2^A\subset \mathrm{G}(3,V)$ of the Lagrangian subspace $A$.  Note that our construction of the double cover $p_X$ is just the restriction of that construction to $D_2^A\cap C_{U_1}$. Indeed, we proceed as shown in the diagram \ref{eq:delta}. The intersection $D_2^A\cap C_{U_1}$ coincides with $D_1^{\bar {A}}$ and $D_3^A\cap C_{U_1}=D_2^{\bar {A}}$.  Also the resolution of singularities  $\tilde {D}_2^A\to D_2^A$ restricts to the resolution of singularities $\tilde {D}_1^{\bar{A}}\to D_1^{\bar{A}}$. The double cover $\tilde {Y}_A\to \tilde {D}_2^A $ restricts to a double cover of $\tilde {D}_1^{\bar{A}}$ branched along  $E$. It hence follows by uniqueness of double cover that the strict transform   of $\tilde {D}_1^{\bar{A}}$ under the double cover $\tilde {Y}_A\to \tilde {D}_2^A $ is isomorphic to  $\tilde {X}_A$.  
\begin{equation}
\label{eq:delta}
 \begin{array}{c}
  \xymatrix@R=.0cm@C=.6cm{& \tilde {X}_A\ar[rr]\ar[rd]\ar[dd]&&\tilde {Y}_A \ar[dd] \ar[rd]& 
 \\ &&X_{\bar{A}}\ar[rr]\ar[dd]&&Y_{A}\ar[dd]\\
& \tilde {D}_1^{\bar{A}}\ar[rr]\ar[rd]&&\tilde {D}_2^{A}\ar[rd]&&\\
C_{U_1}\cap{D}_2^{A}& & {D}_1^{\bar{A}}\ar@{}[ll]|-{=}\ar[rr]&& {D}_2^{A}\ar[r] &\mathrm{G}(3,V)  \\
&&&&&\\
&&{D}_2^{\bar{A}} \ar[uu]\ar[rr]&&{D}_3^{A}\ar[uu]&
}
  \end{array}
\end{equation}

  Therefore $X_{\bar{A}}$ coincides with the strict transform of $D_1^{\bar{A}}$ under the double covering $Y_A\to D_2^A$. Finally $p_X$ is then the restriction  of the double cover $Y_A\to D_2^A$ to $X_{\bar{A}}$:
 
 \begin{prop}\label{compare X_A} Let $[A]\in (\Sigma-(\Sigma_+\cup \Sigma[1]\cup \Gamma)$, 
 let $[U_1]=\PP(A)\cap \mathrm{G}(3,V)$ and $p_Y: Y_A\to D_2^A\subset \mathrm{G}(3,V)$ be $6$-fold double cover ramified over $D_3^A$.  Then  
 $p_X:X_{\bar{A}}\to D_1^{\bar{A}}$ coincides with the restriction of the  double cover $p_Y$ to the preimage $p_Y^{-1}(D_2^A\cap C_{U_1})$.
   \end{prop}




%

We construct in this way a $19$-dimensional family, parametrized by  $$\Sigma-(\Sigma_{+}\cup \Sigma[1]),$$ of  hyperk\"ahler fourfolds admitting polarizations of degree $48$ that define antisymplectic involutions. 

\subsection{The construction} We need to prove that $X_{\bar{A}}$ are hyper-K\"ahler manifolds.

 
   \begin{prop}\label{smooth X_A} Let $[A]\in (\Sigma-(\Sigma_+\cup \Sigma[1]\cup \Gamma))$, let $[U_1]=\PP(A)\cap \mathrm{G}(3,V)$ and let $p_X:X_{\bar{A}}\to D_1^{\bar{A}}$ be the  double cover of (\ref{eq}). Then $X_{\bar{A}}$ is a smooth manifold with trivial first Chern class. 
  \end{prop}
\medskip
 \begin{proof}  The smoothness of $X_{\bar{A}}$ was noted above, so it remains to compute the canonical class.  For this we start with $D_1^{\bar{A}}$, a quartic hypersurface section of the cone $C_{U_1}$, see Lemma \ref{lem-ram}, with quadratic singularities along the surface $D_2^{\bar{A}}$.  Let $\tilde{C}_{U_1}\to C_{U_1}$ be the blowup of the cone  $C_{U_1}$ in the vertex.  Then $\tilde{C}_{U_1}$ is a  $\PP^1$-bundle over $\PP^2\times\PP^2$.  The pullback $h$ to  $\tilde{C}_{U_1}$ of a hyperplane divisor on $\PP^2\times\PP^2$, coincides with the pullback of a hyperplane divisor on $C_{U_1}$.  The pullback of a canonical divisor on $\PP^2\times\PP^2$ is $-3h$, while the relative canonical divisor over $\PP^2\times\PP^2$ is $-h$, so the canonical divisor on $\tilde{C}_{U_1}$ is $-4h$.  By adjunction the fourfold $D_1^{\bar{A}}$ has trivial canonical sheaf.  
 Since the singularities along the surface $D_2^{\bar{A}}$ are quadratic, the canonical divisor of the smooth fourfold $\tilde{D}_1^{\bar{A}}$ is half the class of the exceptional divisor.  The double cover $\tilde{X}_{\bar{A}}$, therefore has canonical divisor equal to the ramification divisor $\tilde{E}$.  On the smooth fourfold $X_{\bar{A}}$, this divisor is blown down, so $X_{\bar{A}}$ has trivial first Chern class.
 \end{proof}

\begin{thm}\label{dwa} There exists a $19$-dimensional family of polarized IHS fourfolds $(X,H)$ such that $|H|$ defines a $2:1$ morphism to 
$\PP^9$ and the image is the intersection of a cone over a Segre product $\PP^2\times \PP^2$ with a special quartic $Q_{\bar{A}}$,
with the branch locus being the surface $S_{\bar{A}}$ defined in Lemma \ref{lem-ram}.  Moreover, each  fourfold in this family admits two Lagrangian fibrations and a polarization with $q=4$.
\end{thm}
Let us be more precise.
Let $A\in \mathrm{LG}_\eta(10,\wedge^3V)$ such that $\PP(A)$ intersects transversally $\mathrm{G}(3,V)$ in one point
(i.e.~$[A]\in (\Sigma-(\Sigma_+\cup \Sigma[1])$).
In this case without, loss of generality  we let $X_{\bar{A}}
\subset Y_A$ be the fourfold defined in (\ref{eq}) and Proposition \ref{smooth X_A}.

 

 In order to prove that $X_{\bar{A}}$ is IHS we need to find a degeneration of $X_{\bar{A}}$ that is birational to the Hilbert scheme of two points on a $K3$ surface. 
 For this, we first consider, for $v\in V$, the $10$-dimensional Lagrangian subspace  
 \[
 F_{[v]}:=\langle v\rangle\wedge(\wedge^2 V)\subset \wedge^3 V.
 \]
 Recall that  
 \[
 \Delta=\{ [A]\in \mathrm{LG}_\eta(10,\wedge^3V)| \quad \exists v\in V \colon \dim A\cap F_{[v]}\geq 3 \}.
 \]
We shall use an $[A]\in \Sigma\cap \Delta$ to find the suitable degeneration.

 By dimension count we infer the following, using the notation of \ref{notation}:
 \begin{lem} \label{dim sigma delta } The set $(\Sigma\cap \Delta)-(\Sigma_+\cup \Sigma[1] \cup \Gamma)$ is nonempty of dimension $18$. 
 \end{lem}
 \begin{proof} By a direct count, we first compute that $\dim (\Sigma\cap \Delta) = 53$. Let $v\in V$ and let $\mathbf{P}\subset \PP(F_{[v]})$ be a plane.  The set of triples $F_{[v]}$,  $\mathbf{P}$ and $[U]\in \mathrm{G}(3,V)\cap \mathbf{P}^{\perp}$ 
 depend on $5 +(3*7)+6=32$ parameters. The set of Lagrangian subspaces $A$ such that $\PP(A)\supset\langle\mathbf{P},[U]\rangle $ is isomorphic to a $\mathrm{LG}(6,12)$ so its dimension is $21$. It follows that $(\Sigma\cap \Delta)$ contains a component  corresponding to general pairs $(\mathbf{P},[U])$.
 We shall compute dimensions of the intersections of this component with $\Sigma_+$, $\Sigma[1]$ and $\Gamma$ separately: 
 \begin{enumerate}
 \item For a general $[A]\in\Sigma_+\cap\Delta$, the linear space $\PP(A)$ contains a pair $(\mathbf{P},[U])$ and a point $[U']\in \mathrm{G}(3,V)\cap \langle\mathbf{P},[U]\rangle^{\perp}$. Since $\mathbf{P}$ and $[U]$ are general, we have  $\mathrm{G}(3,V)\cap\langle\mathbf{P},[U]\rangle= [U]$ and 
 \[
 \dim(\mathrm{G}(3,V)\cap \langle\mathbf{P},[U]\rangle^{\perp})=5.
 \]
  It follows that $[U']\notin \langle\mathbf{P},[U]\rangle$ and the space of choices of $U'$ is 5-dimensional. A dimension count yields $32+5+dim(\mathrm{LG}(5,10))=52$
  \item For a general $[A]\in\Sigma[1]\cap\Delta$, the linear space $\PP(A)$ contains a pair $(\mathbf{P},[U])$ and a line $l\subset \PP(T_{U})$ through $[U]$.   Since $A$ is Lagrangian we get that $l\subset \PP(T_{U})\cap \langle\mathbf{P},[U]\rangle^{\perp}$, and the number of parameters for $A$ given 
  $(\mathbf{P},[U])$ and the line $l$ is $\dim \mathrm{LG}(5,10) =15$

  When $\mathbf{P}$ and $[U]$ are general the number of parameters for $l$ is
  \[
  \dim (\mathrm{G}([U],1,\PP(T_{U})\cap \langle\mathbf{P},[U]\rangle^{\perp}))=5.
  \]
 So summing up we get that  $\Sigma[1]\cap\Delta$ has dimension $32+5+\dim(\mathrm{LG}(5,10))=52$.

\item For a general $[A]\in \Gamma\cap\Sigma\cap\Delta$, the linear space $\PP(A)$ contains a pair $(\mathbf{P},[U])$, and intersects $T_{[U']}$  for some $[U']\in \mathrm{G}(3,V)$ such that  $\dim (\PP(T_{U'})\cap\PP(A))=3$. 
Let  $\dim (\PP(T_{U'})\cap \langle\mathbf{P},[U]\rangle^{\perp})=5+d_1$ and therefore
$\dim (\PP(T_{U'})\cap \langle\mathbf{P},[U]\rangle)=d_1-1$. 
The set of $4$-dimensional subspaces $W_4\subset \wedge^3V$ such that  $\PP(W_4)\subset \PP(T_{[U']})\cap \langle\mathbf{P},[U]\rangle^{\perp}$ and meets $\langle\mathbf{P},[U]\rangle$ in dimension $d_1$ is a Schubert cycle of dimension $8$ for $d_1=0$ and $9$ for $d_1=1$.  On the other hand the dimension of the set of Lagrangian spaces $A$ such that $\PP(A)$ contains $\langle \PP(W_4),\mathbf{P},[U]\rangle$ is 
\[
\dim(\mathrm{LG}(2+d_1,4+2d_1))=\frac{(2+d_1)(3+d_1)}{2}.
\]
 To complete the dimension count we compute the dimension of the set of subspaces $U'$ corresponding to $d_1=0$ and $d_1=1$. 
For $d_1=0$  the set of subspaces $U'$ is an open set in $\mathrm{G}(3,V)$, so the dimension is $9$, so the set of Lagrangian subspaces $A$ in this case has dimension $32+9+3+8=52$. Whereas, for $d_1=1$ the set of subspaces $U'$ such that $\PP(T_{[U']})\cap \langle\mathbf{P},[U]\rangle\not=\emptyset$ has dimension $5$, so 
 the set of Lagrangian subspaces $A$ in this case has dimension $32+5+6+9=52$.

\end{enumerate} 
\end{proof}

 \begin{defi}\label{OGradyassumption}
 Let $v_0\in V$.  We call $\mathrm{LG}_\eta(10,\wedge^3 V)^{v_0}$ the set of Lagrangian subspaces $A\subset \wedge^3 V$ that satisfy the following conditions:
 
 \begin{enumerate}
 \item There exists a codimension 1 subspace $V_0 \subset V$ such that $\wedge^3 V_0\cap A = {0}$.
 \item $v_0\in U$ for at most one $[U]\in \Theta_A$.
 \item If $v_0\in U$ and $[U]\in \Theta_A$, then $A\cap(\wedge^2 U\wedge V)=\wedge^3U$
 \end{enumerate}
 \end{defi}
 Recall that for $A\in \mathrm{LG}(\wedge^3 V)^{v_0}$ O'Grady defined a surface $S_A(v_0)$ as follows \cite{Ogrady-michigan}: By the first two conditions $V=\langle v_0\rangle\oplus V_0$.
 Consider the isomorphism
 \[
 \lambda: \wedge^2V_0\to F_{v_0}=v_0\wedge (\wedge^2V);\qquad \alpha\mapsto v_0\wedge\alpha.
 \]
 Let $K_A^0=\lambda^{-1}(A\cap F_{v_0})\subset \wedge^2V_0$.  Given a volume form on $V_0$, there is an isomorphism $\wedge^3V_0\cong \wedge^2V_0^{\vee} $, and hence the annihilator ${\rm Ann} K_A^0\subset \wedge^3V_0$ defines a linear section $F_A^0=\PP({\rm Ann} K_A^0)\cap Gr(3,V_0)\subset\PP(\wedge^3 V_0)$.
 Now, $K_A^0$ is $3$-dimensional, so the linear section $F_A^0$ has codimension $3$ in $Gr(3,V_0)$ and is a Fano $3$-fold.
 The first assumption in \ref{OGradyassumption}  implies that $A$ is the graph of a linear map 
 \[
 q_A: \wedge^2 V_0\to \wedge^3V_0\subset\wedge^3V
 \]
  such that $q_A(\alpha)=\beta\leftrightarrow (v_0\wedge\alpha+\beta)\in A$.  Since $A$, $F_{[v_0]}$ and $\wedge^3 V_0$ are Lagrangian, the map $q_A$ is symmetric, while ${\rm ker} q_A=K_A^0$, so $q_A $ induces an isomorphism 
  \[
  \wedge^2V_0/K_A^0\to {\rm Ann} K_A^0\subset \wedge^3V_0
  \]
   whose inverse defines a quadratic form 
   \[
   Q^*_A: \beta\mapsto {\rm vol}(\alpha\wedge \beta),\quad  {\rm where}\; q_A(\alpha)=\beta 
   \] 
   on ${\rm Ann} K_A^0$.  The surface $S_A(v_0)$ is the intersection $F_A^0\cap \{Q^*_A=0\}$.
 
 O'Grady proves that if $\Theta_A$ is finite, then $S_A(v_0)$ is reduced and irreducible
 with explicitly described singular locus.
 Moreover, if it has du Val singularities, then the minimal resolution $\overline{S_A(v_0)}\to S_A(v_0)$ is a $K3$ surface \cite[Corollaries ~4.7 and 4.8]{Ogrady-michigan}.

\begin{lem}\label{$K3$node} Let $[A] \in (\Sigma\cap \Delta)-(\Sigma_+\cup \Sigma[1]\cup \Gamma)$ be generic, then there exists a unique $[v]\in \PP(V)$ such that $dim (A\cap F_{[v]})\geq 3$. Moreover, $[A]\in \mathrm{LG}(\wedge^3 V)^{v}$ and the surface $S_A(v)$ is a $K3$ surface with one node.
 \end{lem}
 \begin{proof}
 Consider $F_{[v]}$ for a general  $v\in V$ and a general projective plane $\mathbf{P}\subset \PP(F_{[v]})$ and let $[U]$ be a general point in $\mathrm{G}(3,V)\cap \mathbf{P}^{\perp}$.
  By Lemma \ref{dim sigma delta }, the general Lagrangian space $A$ such that $\PP(A)$ contains $[U]\cup \mathbf{P}$ is an element of  $(\Sigma\cap \Delta)-(\Sigma_+\cup \Sigma[1] \cup \Gamma)$. Clearly $A$ and $v$ then satisfy the assumption of \ref{OGradyassumption} i.e. $A\in \mathrm{LG}_\eta(10,\wedge^3 V)^{v}$.
 
 We need to prove that $S_A(v)$ is a $K3$ surface with one node. We build on the proof of \cite[Proposition~4.6]{Ogrady-michigan}. The Fano threefold $F_A=\mathbf{P}^{\perp} \cap \mathrm{G}(3,V_0)
 $ is smooth and the surface  $S_A(v)$ is a quadric section of $F_A$ that is smooth outside one point. It follows that the singularity is an ordinary double point. 
 \end{proof}

We denote by  $\overline{S_A(v)}\to S_A(v)$ the minimal resolution of singularities on $S_A(v)$.
  Consider the $6$-fold Lagrangian degeneracy locus $Y_A\subset \mathrm{G}(3,V)$ called an EPW cube in \cite {EPWcubes}, defined as 
  \[
  Y_A=\{[U]\in \mathrm{G}(3,V)|\dim A\cap T_U\geq 2\}.
  \]
When $[A]\in  (\Sigma\cap \Delta)-(\Sigma_+\cup \Sigma[1])$ we shall define a rational map   $$\psi\colon S_A(v)^{[3]}\dasharrow Y_A,$$
as in \cite[\S 4]{EPWcubes}.    

First we consider the natural isomorphism:
\[
V^{\vee}=V_0^{\vee}\oplus \langle v_0^*\rangle \to H^0(\mathcal{I}_{S_A(v)}(2));\qquad  v^*+cv_0^*\mapsto q_{v^*}+cq^*_A,
\]
where $Q_{v^*}$ is the restriction to ${\rm Ann} K_A^0$ of the quadratic form on $\wedge^3 V_0$ defined by
\[
Q_{v^*}(\omega)={\rm vol}(\omega(v^*)\wedge \omega).
\]
Let $[\beta_1]$ and $[\beta_2]$ be two points in $S_A(v)$, such that the line $\langle [\beta_1],[\beta_2]\rangle$ is not contained in $S_A(v)$, then 
$H^0(\mathcal{I}_{S_A(v)\cup\langle \beta_1,\beta_2\rangle }(2))$ is a hyperplane in $H^0(\mathcal{I}_{S_A(v)}(2))\cong V^{\vee}$.  Therefore 
\[
\phi\colon S_A(v)^{[2]}\dasharrow \PP(V);\quad ([\beta_1],[\beta_2])\mapsto [H^0(\mathcal{I}_{S_A(v)\cup\langle \beta_1,\beta_2\rangle }(2))]
\]
defines a rational map.  The rational map $\psi\colon S_A(v)^{[3]}\dasharrow \mathrm{G}(3,V)$ is now defined by 
\[
\psi(\beta_1,\beta_2,\beta_3) 
=[\langle \phi(\beta_1,\beta_2),\phi(\beta_1,\beta_3),\phi(\beta_2,\beta_3) \rangle]\in \mathrm{G}(3,V)
\]
For general $A$, both $\phi$ and $\psi $ are morphisms that are $2:1$ onto their image \cite[Proposition 4.1]{EPWcubes}.

We consider a restriction of the map $\psi $ to show


      \begin{prop}\label{prop3.6} Let $[A]\in (\Sigma\cap \Delta)-(\Sigma_+\cup \Sigma[1]\cup \Gamma)$. Let $[U]=\mathrm{G}(3,V)\cap \PP(A)$, and let ${\bar{A}}=A/\wedge^3U$.
      Then $X_{\bar{A}}$ is birational to $\overline{S_A(v)}^{[2]}$.
 \end{prop}
 \medskip
 \begin{proof} 
 Let $U$ be the unique element in $\Theta_A$. Consider the decomposition $V=v_0\oplus V_0$. By  \cite[Corollary~4.7]{Ogrady-michigan}, the $K3$-surface $S:=S_A(v_0)$ 
 in $\mathbb{P}(\wedge^3 V_0)$ is singular in $p:=[\wedge^3 U']$ where $U'$ is the projection of $U$ onto $V_0$. Moreover, by Lemma \ref{$K3$node}, the point $p$ is a node in $S$. 
 Let $\kappa: \overline{S} \to S$ be the blow up giving the resolution of the node $p$. Consider the following rational map $\xi$ defined on pairs of distinct points on $\overline{S}$.
$$\xi: \overline{S}^{[2]}\to \mathrm{G}(3, H^0(\mathcal{I}_S(2));\quad \xi([ p_1,p_2 ] ) = H^0(\mathcal{I}_{S\cup \langle l_{p_1},l_{p_2}\rangle}(2))  \subset \mathrm{G}(3, H^0(\mathcal{I}_S(2)),$$
where 
  $l_{p_i}$ is the line spanned by $p$ and $ \kappa (p_i)$ when  $p\neq \kappa (p_i)$ and  the line in the tangent cone of $p$ corresponding to $p_i$
  when $\kappa (p_i)=p$. 
 Moreover, $\langle l_{p_1},l_{p_2}\rangle$ is the plane spanned by $l_{p_1}$ and $l_{p_2}$.
 \begin{lem}The map $\xi$ is generically 2:1, well defined and unbranched outside a set of codimension 2.   
 \end{lem}
 \medskip
\begin{proof} The proof is analogous to the proofs of Proposition 4.1 and Proposition 4.5 of \cite{EPWcubes}. We need only to observe that $S$ is the intersection of a smooth Fano threefold $F$ with a quadric $Q$ and hence every twisted cubic passing through $p$ is tangent to the quadric $Q$ in $p$. This implies that the involution on $S_A(v)^{[3]}$ given by $\psi$ above restricts to an involution  on the locus of triples on $\overline{S}_A(v)$ containing $p$, and  $\xi$ can be considered as the restriction of $\psi$ to this locus.
\end{proof}
 The next step is to prove that the image of $\xi$ is contained in the cone $C_U=\mathrm{G}(3,V) \cap \mathbb{P}(T_U)=\{[L]\in \mathrm{G}(3,V): \dim (L\cap U)\geq 2\}$. By \cite[Lemma 4.2]{EPWcubes}, we have 
 $$\xi([p_1,p_2])=\psi([p,p_1,p_2])=\langle\phi([p,p_1]),\phi([p,p_2]),\phi([p_1,p_2])\rangle$$ in the above notation. 
 To prove that  $\xi([p_1,p_2])\in C_U$ it is enough to prove that $\phi([p,p_i])\in U$. Let $i=1$.
 We follow the proof of \cite[Proposition 4.1]{EPWcubes}. Indeed let 
 \[
 \wedge^3 U=u_1\wedge u_2 \wedge u_3= v_0\wedge \alpha + v_1\wedge v_2 \wedge v_3 
 \]
  with $v_1,v_2,v_3\in V_0$ and $\alpha\in \wedge^2 \langle v_1,v_2,v_3\rangle,$
  then, by \cite[Corollary~4.7]{Ogrady-michigan}, the singular point of the $K3$ surface $S$ is $p=v_1\wedge v_2 \wedge v_3$. Without loss of generality we may then assume that $p_1=v_1\wedge v_4\wedge v_5$. 
 Then, by \cite[Equation 4.1]{EPWcubes}, we have 
 $\phi([p,p_1])=[{\rm vol}(\alpha \wedge v_1\wedge v_4\wedge v_5)v_0+v_1]$. To check that it is an element of $U$ 
 we compute 
\begin{align*}
&( {\rm vol}(\alpha \wedge v_1\wedge v_4\wedge v_5)v_0+v_1)\wedge (v_0\wedge \alpha + v_1\wedge v_2 \wedge v_3)\\
&=v_1\wedge v_0 \wedge \alpha +vol(\alpha \wedge v_1\wedge v_4\wedge v_5) v_0\wedge v_1\wedge v_2 \wedge v_3.
\end{align*}
The latter is an element of $\wedge^4\langle v_0,v_1,v_2,v_3\rangle$ and  the wedge product with $v_4\wedge v_5$,
\begin{align*}
&(v_1\wedge v_0 \wedge \alpha +{\rm vol}(\alpha \wedge v_1\wedge v_4\wedge v_5) v_0\wedge v_1\wedge v_2 \wedge v_3)\wedge v_4\wedge v_5\\
&=(-v_0 \wedge \alpha\wedge v_1\wedge v_4\wedge v_5+{\rm vol}(\alpha \wedge v_1\wedge v_4\wedge v_5)(v_0\wedge \ldots\wedge v_5)\\
&=(-{\rm vol}(\alpha \wedge v_1\wedge v_4\wedge v_5)+{\rm vol}(\alpha \wedge v_1\wedge v_4\wedge v_5))(v_0\wedge \ldots\wedge v_5)=0,
\end{align*}
so $({\rm vol}(\alpha \wedge v_1\wedge v_4\wedge v_5)v_0+v_1)\in U$.  With the same argument for $i=2$ we conclude that  $\xi([p_1,p_2])\in C_U$, in particular 
$\xi([p_1,p_2])\in D^1_{\bar{A}} \subset C_U$. 

Therefore $X_{\bar{A}}\rightarrow D^1_{\bar{A}}$ and $\xi: S_A^{[2]} \rightarrow D^1_{\bar{A}}$ are two double covers which are well defined and unbranched outside a set of codimension 2. It follows that $X_{{\bar{A}}}$ is birational to $S_A^{[2]}$ as in \cite[\S 5]{EPWcubes} and further still following \cite[\S 5]{EPWcubes} we get $X_{{\bar{A}}}$ is IHS 
and deformation equivalent to a $K3^{[2]}$ for general ${\bar{A}}$. 
 \end{proof}
\begin{rem} The intersection lattice of $\bar{S}_A^{[2]}$, where $\bar{S}_A$ is the minimal resolution of the nodal $S_A$, is the diagonal matrix with entries $10,-2,-2$. After a change of base to $(h_1,h_2,\theta)$ we obtain:
$$\begin{bmatrix}0&2&0\\
2&0&0\\
0&0&-10
\end{bmatrix}$$
We find that the map $\xi$ is given by $h_1+h_2$. Since there is a divisor with self-intersection $-10$ and divisibility $2$ perpendicular to $h_1+h_2$, it follows that $\xi$ contracts a $\PP^2$ to a point (see \cite[\S 5.1]{HT} or \cite[\S 2]{Mon2}).
We can identify this $\PP^2$ as the set of pairs of points on $S_A$ such that the line spanned by these points is contained in the threefold section of $\mathrm{G}(2,5)$ containing $S_A$.
\end{rem}
\begin{rem} We can find another $18$-dimensional subfamily of $\mathcal{U}$ such that the elements are birational to the Hilbert scheme of two point on a $K3$ surface.  
Let us take a $K3$ surface $S$ that is a hyperplane section of a Verra threefold $Z\subset \PP^8$.
The intersection lattice of $S^{[2]}$ is 
$$\begin{bmatrix}2&4&0\\
4&2&0\\
0&0&-2
\end{bmatrix}.$$
After an integral linear change of coordinates the matrix takes the form:
$$\begin{bmatrix}0&2&0\\
2&0&0\\
0&0&-6
\end{bmatrix}$$
with basis $l_1,l_2,\eta'$.
Then $l_1+l_2$ gives a $2:1$ map to an EPW quartic section containing the vertex of the cone and singular at it. We can show that this map contracts two planes $\PP^2$ to this vertex point.

\end{rem}

\section{The second construction- the Hilbert scheme of conics on the Verra 4-fold}\label{F-Verra}
We describe the second construction of elements from $\mathcal{U}$ that is parallel to the construction of Kummer surfaces given in section \ref{hilbert-kummer}.
Let $U_1$ and $U_2$ be three dimensional complex vector spaces, 
 fix moreover a volume form on each space $U_1, U_2^\vee$ such that 
$\wedge^2U_1=U_1^{\vee}$ and $\wedge^2U_2^{\vee}=U_2$, and let $\eta:\wedge^3 U_1\otimes \wedge^3 U_2^\vee\to \C$ be the product volume form.
Let $Y\subset \PP^9$ be the intersection of the cone $C(\PP(U_1)\times \PP(\wedge^2 U_2))\subset \PP(\C\oplus (U_1\otimes \wedge^2 U_2))$ with a quadric hypersurface.
Such a fourfold is a smooth Fano fourfold when $Q$ is chosen generically: we call it a Verra fourfold. We have a $19$-dimensional family of Verra fourfolds.

Note that a Verra fourfold is naturally a double cover of $\PP(U_1)\times \PP(\wedge^2 U_2)$.  Its ramification locus $Z$ is the intersection 
of $Y$ with the hyperplane polar to the vertex of the cone via the quadric $Q$. In terms of coordinates, this means that if coordinates are chosen in such a way that $Q$ is defined by a quadric $\{z^2-Q'=0\}$ then $Z=Y\cap \{z=0\}$. We call $Z$ the Verra threefold associated to $Y$. We shall sometimes also identify $Z$ with the branch locus $\PP(U_1)\times \PP(\wedge^2 U_2)\cap \{Q'=0\}$. 

Notice the following properties of Verra fourfolds.
\begin{lem}\label{smooth Verra} If $Y\subset C(\PP( U_1)\times \PP(\wedge^2 U_2))$ is a smooth Verra fourfold then:
\begin{enumerate}
\item $Y$ does not pass through the vertex of the cone $C(\PP( U_1)\times \PP(\wedge^2 U_2))$;
\item $Y$ contains no quadric threefold;
\item the preimage of each quadric surface $\PP(L^{\vee})\times \PP (M^{\vee})\subset \PP( U_1)\times \PP(\wedge^2 U_2)$ by the double cover $Y\to \PP( U_1)\times \PP(\wedge^2 U_2)$ is irreducible.
\end{enumerate}
\end{lem}
\begin{proof} Clearly $Y$ being a smooth complete intersection of $C(\PP( U_1)\times \PP(\wedge^2 U_2))$ with a quadric cannot pass through the singular point of the cone. 
For (2), if $Y$ contained a quadric threefold, then this threefold would be contained in $C(\PP( U_1)\times \PP( \wedge^2 U_2))$ and hence would be a cone over a quadric surface in $\PP(U_1)\times \PP(\wedge^2 U_2)$. This  leads to a contradiction with (1). 
Finally assume that the preimage of some quadric surface $\PP(L^{\vee})\times \PP (M^{\vee})\subset \PP(U_1)\times \PP(\wedge^2 U_2)$ is reducible. Then it must decompose as the union of two quadric surfaces and the branch locus of the projection onto  $\PP(L^{\vee})\times \PP (M^{\vee})$ is then a double conic. It follows that the branch locus $Z$ of the projection $Y\to \PP(U_1)\times \PP(\wedge^2  U_2)$ meets the $\PP^3$ spanned by $\PP(L^{\vee})\times \PP (M^{\vee})$  in a double conic. By Zak's Tangency theorem \cite{Zak}, this implies that $Z$ is singular and in consequence $Y$ is also singular.
\end{proof}

The linear system of quadrics containing $C(\PP(U_1)\times \PP(\wedge^2 U_2))\subset \PP(\C\oplus (U_1\otimes \wedge^2 U_2))$ is naturally isomorphic to $\PP(U_1 \otimes \wedge^2 U_2)$.  In fact let $w\in U_1\otimes \wedge^2 U_2=U_1\otimes (U_2)^{\vee}$ and $(w_0,w')\in \C\oplus(U_1\otimes \wedge^2 U_2)= \C\oplus(U_1\otimes (U_2)^{\vee})$, then 
\[
Q_w(w_0,w')=\eta(w\wedge w'\wedge w')
\]
is a quadratic form on $\C\oplus(U_1\otimes \wedge^2 U_2)$, and
 the map
\[
U_1 \otimes \wedge^2 U_2\to H^0(\mathcal{I}_{C(\PP(U_1)\times \PP(\wedge^2 U_2))}(2));\quad w\mapsto Q_w
\]
is an isomorphism.  Thus 
\[
I_{Y,2}:= H^0(\mathcal{I}_{Y}(2))\cong \C\oplus (U_1 \otimes \wedge^2 U_2),
\]
and the linear system of quadrics containing $Y\subset \PP^9$ is naturally isomorphic to $\PP(\C\oplus (U_1 \otimes \wedge^2 U_2))$ and is dual to $\PP(\C\oplus (\wedge^2 U_1 \otimes U_2))$. By abuse of notation, we denote also by $Q_{w}$ the quadric hypersurface corresponding to $[w]\in \PP(U_1)\otimes \PP(\wedge^2 U_2)$.

Consider the two natural projections $\pi_i$ of  $Y$ onto  $\PP(U_1)$ and $\PP(\wedge^2 U_2)$ for $i=1,2$ respectively.
We denote by $F(Y)$ the Hilbert scheme of conics on $Y$ of type $(1,1)$ i.e. conics that project to lines by both $\pi_1$ and $\pi_2$. 

Let us now relate the Hilbert scheme $F(Y)$ corresponding to the quadric $Q$ with an EPW quartic section.
Let $C$ be a conic on $Y$, then $C$ spans a plane $\PL_C\subset \PP(\C\oplus (U_1\otimes \wedge^2 U_2))$. Consider the locus $H_C$ of quadrics containing $Y\cup \PL_C$. Clearly
$H_C$ is a hyperplane in the space of quadrics containing $Y$ i.e. naturally a point $H_C\in \PP(\wedge^3 U_1 \oplus (\wedge^2 U_1 \otimes U_2)).$
In this way we defined a morphism
\[
\psi_Q: F(Y)\to \PP(\wedge^3 U_1 \oplus (\wedge^2 U_1 \otimes U_2)).
\]

\begin{prop}The image $\psi_Q (F(Y))$  is isomorphic to an EPW quartic section. 
\end{prop}
\begin{proof}
We first introduce the EPW quartic section that we claim is $\psi_Q (F(Y))$.

For that, choose a coordinate chart $(z,\beta)$ on $\C\oplus (U_1\otimes \wedge^2 U_2)$ in which $Q(z,\beta)=z^2-Q'(\beta)$. Note, that in this case $Q'\cap (\PP(U_1)\times \PP(\wedge^2U_2))$ is the branch locus of  the projection map of $Y$ from the vertex of the cone.

Now, the vector space $((\wedge^2 U_1) \otimes U_2)\oplus (U_1 \otimes (\wedge^2 U_2)) $ is equipped with the symplectic form $\bar{\eta}(\alpha,\beta)=vol(\alpha\wedge \beta)$. Observe that $((\wedge^2 U_1) \otimes U_2)\oplus (U_1 \otimes (\wedge^2 U_2)) $ is then a decomposition into a sum of Lagrangian spaces with respect to $\bar{\eta}$. In particular $\bar{\eta}$ defines the canonical  isomorphism  $((\wedge^2 U_1) \otimes U_2)^{\vee}\simeq (U_1 \otimes (\wedge^2 U_2))$. Now $Q'$ defines a symmetric map $q':(U_1 \otimes (\wedge^2 U_2))\to  (U_1 \otimes (\wedge^2 U_2))^{\vee}=((\wedge^2 U_1) \otimes U_2)$, the graph of this map in $((\wedge^2 U_1) \otimes U_2)\oplus (U_1 \otimes (\wedge^2 U_2)) $ is a Lagrangian space that we call $\bar{A}_{Q'}$.

Since we know that the subset of the Hilbert scheme of conics in $Y$ parameterizing smooth conics is dense in the whole Hilbert scheme of conics the following lemma completes the proof of the proposition.
\begin{lem}\label{hilbert scheme of conics to lagrangian} Let $\mathbf{P}$ be a plane in $\PP(\C\oplus (U_1\otimes \wedge^2 U_2))$ meeting $Y$ in a smooth conic curve $C$ of type $(1,1)$, then the hyperplane $H_\mathbf{P}$ of quadrics containing $Y\cup \mathbf{P}$ is an element of  the EPW quartic section $\bar{D}^{\bar{A}_{Q'}}_1$. Furthermore, if $C$ is contained in the branch locus  $Z$ then $[H_\mathbf{P}]\in
 \bar{D}^{\bar{A}_{Q'}}_2$.
\end{lem}\medskip
\begin{proof}
Let us consider the cone 
\[
C(\PP(U_1)\times \PP(\wedge^2 U_2))\subset \PP(\C\oplus (U_1\otimes (\wedge^2 U_2))) =\PP((\wedge^3 U_2) \oplus (U_1\otimes (\wedge^2 U_2)))
\]
 as 
 \[
 \PP((U_1\oplus U_2)\wedge (\wedge^2 U_2))\cap \mathrm{G}(3,U_1\oplus U_2)\subset \PP(\wedge^3 (U_1\oplus U_2)).
 \] 

Let 
\[
\mathbf{P}=\langle(z_1,\beta_1),(z_2,\beta_2),(z_3,\beta_3)\rangle,
\]
 such that 
 \[
 (z_i,\beta_i)\in C(\PP(U_1)\times \PP(\wedge^2 U_2))\subset \PP((U_1\oplus U_2)\wedge (\wedge^2 U_2))\cap \mathrm{G}(3,U_1\oplus U_2)\subset \PP(\wedge^3 (U_1\oplus U_2)).
 \]
  Since $\mathbf{P}$ meets $\mathrm{G}(3,U_1\oplus U_2)$ in a conic curve, there exists then a basis $u_1,u_2,u_3, v_1,v_2,v_3$ of $U_1\oplus U_2$ 
  such that we have 
  \[
  \beta_1=u_1\wedge v_1\wedge v_2, \beta_2=u_2\wedge v_1\wedge v_3, \beta_3= (u_1+u_2)\wedge v_1\wedge (v_2+v_3).
  \]
   In such basis the coordinate of $H_\mathbf{P}\in \PP(\C\oplus (\wedge^2 U_1)\otimes U_2) $ corresponding to $\wedge^2 U_1\wedge U_2$ is $u_1\wedge u_2\wedge v_1$.
Moreover, by the definition of $\bar{A}_{Q'}$, for each $\beta\in U_1\wedge \wedge^2 U_2$ there exists an $\alpha\in (\wedge^2 U_1)\wedge U_2$ such that $Q'(\beta)=\alpha\wedge \beta$ or equivalently 
$\alpha+\beta\in \bar{A}_{Q'}$. Let us denote by $\alpha_i$ the  elements corresponding to $\beta_i$ under the above. Since $(\alpha_i+\beta_i)\wedge(\alpha_j+\beta_j)=0$ for all $i,j$ we get 
 $\alpha_i\wedge\beta_j=\alpha_j\wedge\beta_i=:c_{i,j}$ for $i\neq j$. 
Now 
$$Q(\lambda_1 (z_i,\beta_i)+\lambda_2 (z_j,\beta_j))=(z_i+\lambda z_j)^2- Q'(\lambda_1\beta_i+ \lambda_2 \beta_j)=$$
$$=(z_i+\lambda z_j)^2- (\lambda_1\alpha_i+ \lambda_2 \alpha_j)\wedge(\lambda_1\beta_i+ \lambda_2 \beta_j)=\lambda_1^2 q((z_i,\beta_i))+\lambda_2^2 q((z_j,\beta_j))+2 \lambda_1\lambda_2 (z_iz_2-c_{i,j}).$$
But $Q((z_i,\beta_i))=0$ by assumption, so
$$Q(\lambda_1 (z_i,\beta_i)+\lambda_2 (z_j,\beta_j))= 2\lambda_1\lambda_2 (z_iz_2-c_{i,j}).$$
Now 
$$(t_0Q+t_1 Q_{(u_1\wedge u_2\wedge v_1)^*}) (\lambda_1 (z_i,\beta_i)+\lambda_2 (z_j,\beta_j))=2t_0 \lambda_1\lambda_2 (z_iz_j-c_{i,j})+ 2t_1\lambda_1\lambda_2.$$
It follows that the $H_{\mathbf{P}}=[(z_iz_j-c_{i,j},  u_1\wedge u_2\wedge v_1) ]\in \PP(\C\oplus U_1\wedge (\wedge^2 U_2))$ which means, in particular, that:
$$z_1z_2-c_{1,2}=z_1z_3-c_{1,3}=z_2z_3-c_{2,3}=c_{\mathbf{P}}.$$
$H_{\mathbf{P}}$ is also an element of the cone $C(\PP(\wedge^2 U_1)\times \PP(U_2))\subset \PP(\C\oplus \wedge^2 U_1\wedge U_2)$

The corresponding $\bar{T}_{H_{\mathbf{P}}}$ is described by 
$$\{\omega\in ((\wedge^2 U_1)\otimes U_2) \oplus ((\wedge^2 U_2)\otimes U_1) | \omega\wedge u_1\wedge u_2= \omega\wedge u_1\wedge  (v_1+c_{\mathbf{P}} u_3)=\omega\wedge u_2\wedge  (v_1+c_{\mathbf{P}} u_3)=0\}.$$
We shall prove that $\bar{A}_{Q'}\cap \bar{T}_{H_{\mathbf{P}}}\neq 0$. We know that $\sum_{i=1}^3 \lambda_i (\alpha_i+\beta_i)\in \bar{A}_{Q'}$ for $\lambda_i\in \C$,  it is therefore enough to prove that the following system of equations has a nonzero solution $(\lambda_1,\lambda_2,\lambda_3)$:
\begin{equation*}
\begin{split}
E_1(\lambda_1,\lambda_2,\lambda_3):=(\sum_{i=1}^3 \lambda_i (\alpha_i+\beta_i))\wedge u_1\wedge u_2=0\\
E_2(\lambda_1,\lambda_2,\lambda_3):=(\sum_{i=1}^3 \lambda_i (\alpha_i+\beta_i))\wedge u_1\wedge (v_1+c_{\mathbf{P}} u_3)=0\\
E_3(\lambda_1,\lambda_2,\lambda_3):=(\sum_{i=1}^3 \lambda_i (\alpha_i+\beta_i))\wedge u_2\wedge (v_1+c_{\mathbf{P}} u_3)=0
\end{split}
\end{equation*}
Observe now that $E_1(\lambda_1,\lambda_2,\lambda_3)\equiv 0$  since both $\alpha_i\wedge u_1\wedge u_2=0$ and $\beta_i\wedge u_1\wedge u_2=0$.
Furthermore, we have:
$$E_2(\lambda_1,\lambda_2,\lambda_3)\wedge u_1=E_2(\lambda_1,\lambda_2,\lambda_3)\wedge u_2=E_2(\lambda_1,\lambda_2,\lambda_3)\wedge u_3=E_2(\lambda_1,\lambda_2,\lambda_3)\wedge v_1=0,$$
as well as 
$$E_3(\lambda_1,\lambda_2,\lambda_3)\wedge u_1=E_3(\lambda_1,\lambda_2,\lambda_3)\wedge u_2=E_3(\lambda_1,\lambda_2,\lambda_3)\wedge u_3=E_3(\lambda_1,\lambda_2,\lambda_3)\wedge v_1=0.$$
Finally 
\begin{align}
E_2(\lambda_1,\lambda_2,\lambda_3)\wedge v_2= z_1^2 \lambda_1 +z_1z_2 \lambda_2+z_1z_3 \lambda_3,\\
E_3(\lambda_1,\lambda_2,\lambda_3)\wedge v_3= z_1z_2 \lambda_1 +z_2^2 \lambda_2+z_2z_3 \lambda_3,\\
(E_2(\lambda_1,\lambda_2,\lambda_3)+E_3(\lambda_1,\lambda_2,\lambda_3))\wedge (v_2+v_3)= z_1z_3 \lambda_1 +z_2 z_3 \lambda_2+z_3^2 \lambda_3,
\end{align}
are proportional, so the above equations reduce to two linear equations in the $\lambda_i$:
$$E_2(\lambda_1,\lambda_2,\lambda_3)\wedge v_3=0$$
and one of the above 3 proportional equations.
It follows that the linear system has rank 2 and therefore admits a nontrivial solution implying 
$\dim (\bar{A}_{Q'}\cap \bar{T}_{H_{\mathbf{P}}})\geq 1$, which proves the first part of the lemma.

It remains to prove that the image of the Hilbert scheme of $(1,1)$-conics contained in the ramification locus $Y\cap \{z=0\}$ of the projection maps to $\bar{D}^{\bar{A}_{Q'}}_2$. Clearly points on such conics satisfy $z_i=0$ and the three proportional equations above are then trivial, hence the system has two-dimensional solution i.e. $\dim (\bar{A}_{Q'}\cap \bar{T}_{H_{\mathbf{P}}})\geq 2$.
\end{proof}
\end{proof}

We shall now describe the Stein factorization of the morphism
\[
\psi_Q: F(Y)\to \bar{D}^{\bar{A}_{Q'}}_1.
\]
Consider the diagram:
$$\begin{CD} \PP(\F) @>\pi>>\PP(\wedge^2 U_1)\times \PP(U_2)\\
@V f VV @.\\
\PP(\C\oplus (\wedge^2 U_1\otimes U_2)) @. \end{CD}$$
where 
\begin{equation}\label{F} \F=\pi_1^{\vee}( (\oo_{\PP(\wedge^2 U_1)}(1)) \otimes \pi_2^{\vee} (\oo_{\PP(U_2)}(1)))\oplus \C ) \end{equation} is a vector bundle on  $\PP(\wedge^2 U_1)\times \PP(U_2)$ such that $\pi$ is the projection. Moreover, $f$ is given by $\oo_{\PP(F)}(1)$ such that the image of $f$ is the cone over $\PP(\wedge^2 U_1)\times \PP(U_2)$ and $f$ is the blow-up of the vertex with exceptional divisor $E$.

Consider the rank $5$ bundle $\G$ over $\PP(\wedge^2 U_1)\times \PP(U_2)$, such that for $(L,M)\in \PP(\wedge^2 U_1)\times \PP(U_2)$ 
the fiber $\G_{(L,M)}$ is $$\C\oplus(L^{\vee}\otimes M^{\vee})\subset \C\oplus ( U_1\otimes \wedge^2 U_2).$$ 
There is a natural restriction map $I_{Y,2}\to Sym^2\G^{\vee}_{(L,M)}$.  When $Y$ contains no quadric threefold, this map has rank $2$, and the image is a pencil of quadric threefolds that defines a complete intersection that we denote by $D_{(L,M)}$.   Thus for each $(L,M)\in \PP(\wedge^2 U_1)\times \PP(U_2)$  there is a
natural surjective restriction map 
\[
\pi_{(L,M)}:I_{Y,2}\to I_{D_{(L,M)},2}:=H^0(\PP(\G_{(L,M)}),\mathcal{I}_{D_{(L,M)}}(2))\subset Sym^2\G^{\vee}_{(L,M)}.
\]
For each element $\mathfrak{Q}\in I_{D_{(L,M)},2}$, let $H_{\mathfrak{Q}}\subset I_{Y,2}$ be the hyperplane of quadrics whose image in $I_{D_{(L,M)},2}$ is proportional to ${\mathfrak{Q}}$.
We define degeneracy loci
\[
D^Q_r=\{ ([H_{\mathfrak{Q}}],(L,M))| {\mathfrak{Q}}\in I_{D_{(L,M)},2},
\;{\rm rank}({\mathfrak{Q}})\leq 5-r  \}\subset \PP(\C\oplus (\wedge^2 U_1\otimes U_2))  \times \PP(\wedge^2 U_1)\times \PP(U_2).
\]
Notice that  $D^Q_r\subset \PP(\F)$. 
Consider the projections
\[
f|_{D^Q_r}:D^Q_r\to  \PP(\C\oplus (\wedge^2 U_1\otimes U_2)); \quad ([H_{\mathfrak{Q}}],(L,M))\mapsto [H_{\mathfrak{Q}}] \quad r=1,2.
\]
We claim 

\begin{lem}
$f(D^Q_1)=\psi_Q(F(Y))$.
\end{lem}
\begin{proof}
A $(1,1)$-conic $C$ is mapped to a unique pair of lines $\PP(L^{\vee})\subset \PP( U_1)$ and $\PP(M^{\vee})\subset\PP(\wedge^2 U_2)$, 
and is therefore contained in a unique complete intersection $D_{(L,M)}$, so if ${\mathfrak{Q}}\in I_{D_{(L,M)},2}$ is the quadric threefold that contains the plane $P_C$ spanned by $C$, then $H_C\subset I_{Y,2}$ is the hyperplane of quadrics that contain $P_C$, i.e. $\psi{[C]}=f([H_{\mathfrak{Q}}],(L,M))$.  On the other hand, if  ${\mathfrak{Q}}\in I_{D_{(L,M)},2}$ is singular, then, by Lemma \ref{smooth Verra}(4), it has rank $4$ or $3$, and the planes in ${\mathfrak{Q}}$ intersect $D_{(L,M)}$ in conics that are $(1,1)$-conics on $Y$.
\end{proof}

Next, we claim that $f$ restricted to $D^Q_1$ has an inverse $f^{-1}:f(D^Q_1)\to \PP(\F)$. Indeed,
the quadrics in the ideal of $Y$ define a rational map: 
\[
\PP(\C\oplus ( U_1\otimes \wedge^2 U_2))\dashrightarrow \PP(\C\oplus (\wedge^2 U_1\otimes U_2)).
\]
The preimage of a point $p \in C\PP(\wedge^2 U_1)\times \PP(U_2)$, outside the vertex,  is the union 
\[
Y\cup {\mathfrak{Q}}_p,
\]
 where ${\mathfrak{Q}}_p\in I_{D_{(L,M)},2}$, and, by abuse of notation, at the same time ${\mathfrak{Q}}_p$ is a quadric threefold  in $\PP(C\oplus (L^{\vee}\otimes M^{\vee}))$.  Therefore, the quadrics in the hyperplane $H_{\mathfrak{Q}}$  with ${\mathfrak{Q}}\in I_{D_{(L,M)},2}$ define the pair $(L,M)$ and hence also ${\mathfrak{Q}}$, so $f$ has an inverse.
 
 We choose coordinates such that $Y$ is the intersection of $C(\PP(\wedge^2 U_1)\times \PP(U_2))$ with a quadric $\{ z^2-Q'=0\}$, where $\{Q'=0\}$ is a cone with vertex at the vertex of $C(\PP(\wedge^2 U_1)\times \PP(U_2))$, and $z$ is nonzero at the vertex.  

 The pencil  $I_{D_{(L,M)},2}$ contains in general $5$ rank $4$ quadrics.  One is the rank $4$ quadric $C(\PP(L^{\vee})\times \PP(M^{\vee}))$.  The planes in this quadric intersect $D_{(L,M)}$ in conics that are contracted, by the projection to either $\PP( U_1)$ or $\PP(\wedge^2 U_2)$, so these are not $(1,1)$-conics.  
 A plane in any of the other singular quadrics in $I_{D_{(L,M)},2}$, intersects $D_{(L,M)}$ in a $(1,1)$-conic.  When ${\mathfrak{Q}}\in I_{D_{(L,M)},2}$ has rank $4$,
the fiber  $\psi_Q^{-1} ([H_{\mathfrak{Q}}])$ is therefore two $\PP^1$'s of conics defined on $D_{(L,M)}$ by the two pencils of planes in ${\mathfrak{Q}}$.
The two pencils coincide precisely when ${\mathfrak{Q}}$ has rank $3$. 

The double cover $Y\to \PP(U_1)\times \PP(\wedge^2 U_2)$ is branched along the Verra threefold $Z=Y\cap \{z=0\}$.  It 
defines an involution on $Y$, that for each $(L,M)$ restricts to an involution on $D_{(L,M)}$ and on each threefold quadric ${\mathfrak{Q}}$, where
 ${\mathfrak{Q}}\in I_{D_{(L,M)},2}$.
In particular, when ${\mathfrak{Q}}$ has rank $4$, the two pencils of planes in the quadric  are interchanged by this  involution. 

Finally, when ${\mathfrak{Q}}\in I_{D_{(L,M)},2}$ has rank $3$, then  $D_{(L,M)}$ is singular in two points on the vertex of ${\mathfrak{Q}}$.  So the double cover 
$$D_{(L,M)}\to \PP(L^{\vee})\times \PP(M^{\vee})\subset \PP( U_1)\times \PP(\wedge^2 U_2)$$
is branched along a curve with two singular points, i.e. a pair of conics $C\cup C'$, corresponding to the fixed points of the involution on the pencil of planes in ${\mathfrak{Q}}$.  The pair of conics $C\cup C'$ lies in the hyperplane $\{z=0\}$, i.e. in the Verra threefold $Z=Y\cap \{z=0\}$.   Conversely,  a $(1,1)$-conic $C$ in $Z$ is mapped to a pair of lines $\PP(L^{\vee})$ and $ \PP(M^{\vee})$ and is a component of the ramification locus of the double cover $D_{(L,M)}\to \PP(L^{\vee})\times \PP(M^{\vee})$.  The other component $C'$ is also a 
$(1,1)$-conic contained in $Z$ and $C$ and $C'$ intersect in a scheme of length $2$.  The complete intersection $D_{(L,M)}$ is singular along this scheme, which is the intersection of  the vertex of a rank $3$ quadric ${\mathfrak{Q}}\in I_{D_{(L,M)},2}$ and $D_{(L,M)}$.  

Thus, we have identified the  set of pairs $([H_{\mathfrak{Q}}],(L,M))\in D^Q_2$ where ${\mathfrak{Q}}$ has rank $3$ with the set of pairs of $(1,1)$-conics $C\cup  C'$ in $Z$ that intersect in a scheme of length $2$.  
By Lemma \ref{hilbert scheme of conics to lagrangian}, we 
 infer $f(D^Q_2)={D}^{\bar{A}_{Q'}}_2$. 
Summing up, we have precisely described the Stein factorization of the map  $\psi_Q$.
\begin{prop} The Stein factorization of $\psi_Q$ is 
$$ \psi_Q=\phi \circ \rho$$
with $\phi:F(Y)\to X_Q$ a $\mathbb{P}^1$ fibration  and $\rho: X_Q\to {D}^{\bar{A}_{Q'}}_1 $ a 2:1 cover branched  precisely in
${D}^{\bar{A}_{Q'}}_2$.
\end{prop}

Moreover, we have proven the following relation between the singular locus of an EPW quartic and its associated Verra threefold.
\begin{prop}\label{Z_Q} Let $Y=Q\cap C(\PP(\wedge^2 U_1)\times \PP(U_2))$  be a general Verra fourfold and let $Z=(\PP(\wedge^2 U_1)\times \PP(U_2))\cap Q' = Y\cap \{z=0\}$ be its  associated Verra threefold. Then  the map $$\psi_Q|_{F(Z)}\colon F(Z)\to C(\PP(\wedge^2 U_1)\times \PP(U_2))$$ is an \'etale $2:1$ map to the set $D_2^{\bar{A}_{Q'}}\subset C(\PP(\wedge^2 U_1)\times \PP(U_2))$. Thus the singular set of a general EPW quartic section admits an \'etale double cover
being the Hilbert scheme of conics on the corresponding Verra threefold $Z$.
\end{prop}

Finally, Theorem \ref{main2} appears also as a direct consequence of the above arguments.
 
 \begin{proof}[Proof of Theorem 0.2] Let $X_{\bar{A}}\in \mathcal{U}$, then $X_{\bar{A}}$ is a double cover of  $D_1^{\bar{A}}$ for some Lagrangian $\bar{A}\subset ((\wedge^2 U_1) \otimes U_2)\oplus (U_1\otimes (\wedge^2 U_2)) $. Let $Q_{\bar{A}}\subset \mathbb{P}((U_1\otimes (\wedge^2 U_2))$ be the corresponding quadric and $Z_{\bar{A}}$ be the corresponding Verra threefold and $Y_{\bar{A}}$ the corresponding Verra fourfold. Then both $X_{Q_{\bar{A}}}$ and $X_{\bar{A}}$ appear as double covers of $D_1^{\bar{A}}$ branched in $D_2^{\bar{A}}$ hence are isomorphic. It follows that $X_{\bar{A}}$ is the base of a $\mathbb{P}^1$
fibration on $F(Y_{\bar{A}})$.  For the converse we just need to recall that there is a 1:1 correspondence between general Lagrangian subspaces $\bar{A}$ and general quadrics $Q_{\bar{A}}$. \end{proof}
 
\begin{rem}\label{U2} Observe that if $V_6$ a 6-dimensional vector space and  $[A]\in \mathrm{LG}(10,\wedge^3 V_6)$ such that $\mathbb{P}(A)\cap \mathrm{G}(3,V_6)=\{[U_1]\}$ then to $A$ we associate a unique EPW quartic section $D^1_{\bar{A}}$ and  also a unique Verra fourfold $V_A$. The Verra fourfold appears as follows.
  First, for a fixed choice of $[U_2]\in \mathrm{G}(3,V_6)$ such that $U_2\cap U_1=\{0\} $ consider 
$$q_{A,U_2}: T_{[U_2]}/<[U_2]> \to (T_{[U_1]}/<[U_1]>)=(T_{[U_2]}/<[U_2]>)^{\vee}$$ the symmetric map whose graph is $A/<[U_1]>$ and let $Q_{A,U_2}$ be the corresponding quadric. 
Let $C_{U_2}=T_{[U_2]}\cap \mathrm{G}(3,U)$ and $P_{U_2}=\PP(\wedge^2 U_2)\times \PP(U_1)$ be the corresponding Segre embedding $$\mathbb{P}^2\times \mathbb{P}^2\subset T_{[U_2]}/<[U_2]> \simeq (T_{[U_1]}/<[U_1]>)^{\vee}.$$
Define $Z_{A,U_2}=P_{U_2}\cap Q_{A,U_2}$ the Verra threefold associated to $A$ and $U_2$ and $V_{A,U_2}$ the corresponding Verra fourfold. We claim that in fact $Z_{A,U_2}$ (and in consequence $V_{A,U_2}$) is independent from the choice of $U_2$. Indeed, if we choose a different  $[U'_2]\in \mathrm{G}(3,V_6)$ then we have a canonical isomorphism $T_{[U'_2]}/<[U'_2]> \simeq(T_{[U_1]}/<[U_1]>)^{\vee}\simeq  T_{[U_2]}/<[U_2]>$ induced by the symplectic form  and under this identification we have $Q_{A,U_2}-Q_{A,U'_2}\in H^0(I_{P_{U_2}}(2))=H^0(I_{P_{U'_2}}(2))$.

\end{rem}

 \subsection{Two Lagrangian fibrations}\label{lagr}

Observe that a general double EPW quartic section $X$  admits two fibrations.
Indeed, consider the composition of maps $X_{\bar{A}}\to D_1^{\bar{A}} \subset C(\PP(\wedge^2 U_1)\times \PP(U_2))=C(\PP^2\times\PP^2)$, with $D_1^{\bar{A}}$ the EPW quartic section defined by the Lagrangian subspace ${\bar{A}}\subset (\wedge^3 U_1)^{\bot}/(\wedge^3 U_1)$.
The projections to the factors of $\PP(\wedge^2 U_1)\times \PP(U_2)$ induces two fibrations $\pi_1$ and $\pi_2$. Since $X_{\bar{A}}$ is IHS the fibers are abelian surfaces.
Let us study these fibrations in more details. We shall consider the fibration of the EPW quartic section $D_1^{\bar{A}} \subset C( \PP(\wedge^2 U_1)\times \PP(U_2))$.
\begin{prop} The general fibers of the two natural fibrations $\pi_1: D_1^{\bar{A}} \to \PP(\wedge^2 U_1)$ and $\pi_2: D_1^{\bar{A}} \to \PP(U_2)$ of the EPW quartic section $D_1^{\bar{A}}$ are Kummer quartic surfaces.
\end{prop}
\begin{proof} We consider the fibers of the second projection $\pi_2$, the fibers of $\pi_1$ are treated similarly. 
Let $v\in \PP(U_2)$ be generic. 
Denote by $\PP(V_2)\subset \PP(\wedge^2 U_2)$ the line dual to $v$.
This induces a subset $C(\PP(U_1)\times \PP(V_2))\cap Q_{\bar{A}}$ of the corresponding Verra fourfold  $C( \PP(U_1)\times \PP(\wedge^2 U_2))\cap Q_{\bar{A}}$.

We can identify the fiber $\pi_2^{-1}(v)$ as the image by $\psi_{Q_{\bar{A}}}$ 
of the conics contained in
$C(\PP(U_1)\times \PP(V_2))\cap Q_{\bar{A}}$.
It follows from Proposition \ref{1.10} that this fiber is a Kummer surface.
\end{proof}
\begin{rem} Note that from the adjunction formula $\pi_1$ and $\pi_2$ induces two Lagrangian fibrations  
on $X_{\bar{A}}$. The Kummer surfaces above can be seen as quotient of the Abelian surfaces in the fibers. 
\end{rem}
\begin{rem} Note that also the original description of the EPW quartic section as a Lagrangian degeneracy locus induces naturally a description of the Kummer quartic fibers as Lagrangian degeneracy loci in $\mathbb{P}^3$. That description is consistent with Lemma \ref{lagrangian Kummer} in the following sense.  We analyze both fibrations separately:

\textbf{(1)}The fibers of $\pi_2: D_1^{\bar{A}} \to \PP(U_2)$.
We know that 
\[
D_1^{\bar{A}} \subset C(\mathbb{P}(\wedge^2 U_1)\times \mathbb{P}(U_2))
\]
 hence a fiber $D_{u_2}$  of the projection $\pi_2: D_1^{\bar{A}} \to \PP(U_2)$ of a point $[u_2]\in \PP(U_2)$ is the intersection of 
 \[
 P_{[u_2]}=\PP((\wedge^3 U_1) \oplus ((\wedge^2 U_1)\otimes u_2)))\cap \mathrm{G}(3,U_1\oplus U_2)=C(\mathbb{P}(\wedge^2 U_1)),
 \]
with the Lagrangian degeneracy locus $D_1^{\bar{A}}$: 

\[
D_{u_2}=P_{[u_2]}\cap\bar{D}_1^{\bar{A}}=\{[U]\in P_{[u_2]}\cap \mathrm{G}(3,\wedge^2 U_1\oplus U_2)| \dim(\bar{T}_U\cap \bar{A})\geq 1\}. 
\]

Let 
\[
K_4=U_1\oplus\langle u_2\rangle.
\]
 then 

\[
\wedge ^3K_4=\wedge^3 U_1 \oplus ((\wedge^2 U_1)\otimes u_2)))\subset \wedge^3 V.
\]

Thus, for all $[U]\in P_{u_2}$ we have $\wedge^3U\subset \wedge ^3K_4$ and 
$$T_U=((\wedge^2 U )\wedge V )\supset (\wedge^3 K_4)$$
Since $T_U$ is Lagrangian with respect to the wedge product form on $\wedge^3V$, we have $T_U\subset (\wedge^3 K_4)^{\perp}$. 

Consider the $12$-dimensional quotient space $(\wedge^3 K_4)^\perp /(\wedge^3 K_4)$, with the nondegenerate $2$ form induced by the wedge product form. Then 
$$T_{U}/(\wedge^3 K_4) \subset (\wedge^3 K_4)^\perp /(\wedge^3 K_4)$$ is a Lagrangian subspace.
The Lagrangian subspace $A\subset \wedge^3 V$ contains $\wedge^3 U_1$ so has a Lagrangian quotient $\bar{A}\subset (\wedge^3 U_1)^{\perp}/(\wedge^3 U_1)$. It follows that the image $\bar{A}_{K_4}:= Im \varphi $ of the natural projection 
$$\varphi: \bar{A}\cap (\wedge^3 K_4)^{\perp}/(\wedge^3 U_1) \to (\wedge^3 K_4)^{\perp}/(\wedge^3 K_4)$$
is an isotropic subspace.  

By the genericity of $\bar{A}$, we have $\bar{A}\cap  ((\wedge^3 K_4)/(\wedge^3 U_1))=0$ (for every $u_2\in U_2$), so $\dim (\bar{A}_{K_4})=6$ and $\bar{A}_{K_4}$ is Lagrangian (for every $u_2\in U_2$). 

Finally for $[U]\in P_{u_2}$ i.e. $U\subset K_4$,
\[
[U]\in D_{u_2}\iff \dim(\bar{T}_U\cap \bar{A})\geq 1 \iff \dim((T_U/(\wedge^3 K_4))\cap \bar{A}_{K_4})\geq 1,
\]
i.e. the fiber $D_{u_2}$ is a Lagrangian degeneracy locus associated to the family of Lagrangian subspaces 
\[
\{T_U/(\wedge^3 K_4)|U\subset K_4\}
\]
and the fixed space $\bar{A}_{K_4}$ as Lagrangian subspaces of 
$(\wedge^3 K_4)^{\perp}/(\wedge^3 K_4)$.

With a choice of decomposition $V=K_4\oplus K_2$ we may identify 
$$(\wedge^3 K_4)^{\perp}/(\wedge^3 K_4)= (\wedge^2 K_4) \otimes K_2\subset \wedge^3 V,$$
and identify the $6$-dimensional subspace  $\bar{A}_{K_4}$ with a Lagrangian subspace in $(\wedge^2 K_4) \otimes K_2$, finally we identify:
\[
T_U=(\wedge^2 U) \otimes K_2\subset (\wedge^2 K_4) \otimes K_2.
\]  
In this context $D_{u_2}\subset  \PP(K_4^{\vee}) $ is  the first degeneracy locus 
$$\{[U]\in \PP(K_4^{\vee}) | \dim ((\wedge^2 U) \otimes K_2)\cap \bar{A}_{K_4})\geq 1\} .$$ 
This degeneracy locus was described in Section \ref{kummer} as a Kummer quartic singular in $16$ points given by:
$$\{[U]\in \PP(K_4^{\vee}) | \dim ((\wedge^2 U) \otimes K_2)\cap \bar{A}_{K_4})\geq 2\}.$$

\textbf{(2)} Consider next a fiber of the first projection $\pi_1: D_1^{\bar{A}} \to \PP(\wedge^2 U_1)$  from the Lagrangian degeneracy locus
\[
D_1^{\bar{A}} \subset C(\mathbb{P}(\wedge^2 U_1)\times \mathbb{P}(U_2)).
\]
Let $M_2\subset U_1$ be a $2$- dimensional subspace and denote by  $D_{M_2}$ the fiber $\pi_1^{-1}([\wedge^2M_2])$. 

Let $U\supset M_2$ be $3$-dimensional subspace of $V$, then 
\[
\wedge^3 U\in\wedge^2M_2\wedge V\subset \wedge^3V.
\]
  Thus  we may identify the sets $\{[U]\in \mathrm{G}(3,V)| U\supset M_2\}=P_{[M_2]}$ where 
 \[
P_{[M_2]}:=\PP(\wedge^2M_2\wedge V)\subset \PP(\wedge^3V)
 \]
The fiber $\pi_1^{-1}([\wedge^2M_2])$ is then 
\[
D_{M_2}=D_1^{\bar{A}}\cap P_{[M_2]}.
\]



 
Notice that for each $[U]\in P_{M_2}$,
\[
\wedge^2 M_2\wedge V\subset T_U=((\wedge^2 U )\wedge V ).
\]
In particular 
\[
\wedge^2 M_2\wedge V\subset T_{U_1}=\wedge^3 U_1 \oplus ((\wedge^2 M_2)\otimes U_2)\subset \wedge^3 V.
\]
Since $T_U$ is Lagrangian with respect to the wedge product form on $\wedge^3V$, we have $T_U\subset (\wedge^2 M_2\wedge V)^{\perp}$.
Consider the $12$-dimensional quotient space 
\[
(\wedge^2 M_2\wedge V)^{\perp}/(\wedge^2 M_2\wedge V),
\]
with the nondegenerate $2$-form  induced by the wedge product.
Then 
\[
T_U/(\wedge^2 M_2\wedge V)\subset (\wedge^2 M_2\wedge V)^{\perp}/(\wedge^2 M_2\wedge V)
\] 
is a Lagrangian subspace.
The Lagrangian subspace $A\subset \wedge^3V$ contains $\wedge^3 U_1$, so has a Lagrangian quotient $\bar{A}\subset \wedge^3 U_1^{\bot}/\wedge^3 U_1$.
So the subspace $\bar{A}\cap ((\wedge^2 M_2\wedge V)^{\bot}/(\wedge^3 U_1))$  of $\bar{A}$ is isotropic. 
The projection 
\[
\phi:\bar{A}\cap ((\wedge^2 M_2\wedge V)^{\bot}/(\wedge^3 U_1))\to (\wedge^2 M_2\wedge V)^{\perp}/(\wedge^2 M_2\wedge V)
\] 
therefore has an image 
\[
Im\; \phi:=\bar{A}_{M_2}\subset(\wedge^2 M_2\wedge V)^{\perp}/(\wedge^2 M_2\wedge V).
\]
which is isotropic.
 By the genericity of $\bar{A}$ we have $\bar{A}\cap  ((\wedge^2 M_2\wedge V)/(\wedge^3 U_1))=0$ (for every $M_2\subset U_1$), so $\dim \bar{A}_{M_2}=6$, and
$\bar{A}_{M_2}$ is in fact a Lagrangian subspace for every $M_2\subset U_1$.



Finally for $[U]\in P_{M_2}$ i.e. $U\supset M_2$,
\[
[U]\in D_{M_2}\iff \dim(\bar{T}_U\cap \bar{A})\geq 1 \iff \dim((T_U/((\wedge^2 M_2\wedge V))\cap \bar{A}_{M_2})\geq 1,
\]
i.e. the fiber $D_{M_2}\subset  P_{[M_2]}$ is the first Lagrangian degeneracy locus associated to the family of Lagrangian subspaces 
\[
\{T_U/(\wedge^2 M_2\wedge V)|U\supset M_2\}
\]
and the fixed space $\bar{A}_{M_2}$ as Lagrangian subspaces of 
$(\wedge^2 M_2\wedge V)^{\perp}/(\wedge^2 M_2\wedge V)$, i.e.
\[
D_{M_2}=\{[U]\in P_{[M_2]} | \dim (\wedge^2U\wedge V/((\wedge^2 M_2) \wedge V))\cap \bar{A}_{M_2})\geq 1\} .
\]

If we set $V_2=M_2$ and decompose  $V=V_2\oplus V_4$, then 
\[
(\wedge^2 M_2\wedge V)^{\perp}/(\wedge^2 M_2\wedge V)\cong (\wedge^2 V_2\wedge V_4)^{\bot}/\wedge^2 V_2\wedge V_4\cong V_2\otimes \wedge^2V_4.
\]
On the one hand we can identify the space $\{U\supset M_2\}$ with $\{\langle v\rangle=U\cap V_4\}$.
If $A'$ is the Lagrangian subspace corresponding to $\bar{A}_{M_2}$ via these isomorphisms, then the fiber $\pi_1^{-1}([\wedge^2M_2])$ is isomorphic to 
\[
D_{M_2}\cong\{[v]\in \PP(V_4) | {\rm rank} A'\cap (V_2\otimes V_4\wedge v)\cap A')\geq 1\} .
\]
This degeneracy locus is the Lagrangian degeneracy locus $D_1^{A'}$ of Lemma \ref{lagrangian Kummer} and is a Kummer quartic surface singular in $16$ points.

\end{rem}
\begin{cor}The general fibers of the fibrations $\mathbb{P}(\wedge^2 U_1)\leftarrow X_{\bar{A}}\to \mathbb{P}(U_2)$ are abelian surfaces.  The projections factor 
through the double cover $X_{\bar{A}}\to D^{\bar{A}}_1$, which for each fiber is the double cover of a Kummer quartic surface branched in its 16 singular points. 
\end{cor}

\section{The third construction-moduli space of twisted sheaves}\label{known}
It was observed by G.~Mongardi, that the generic element from the family $\mathcal{U}$ can be constructed as the moduli space of twisted sheaves on a $K3$ surface of degree 2.
We know that the generic element from the family $X\in \mathcal{U}$ admits two Lagrangian fibrations $\pi_i\colon X\to \PP^2$. 
In particular, we obtain two sextic curves as discriminant curves of the fibrations on the bases.
The double cover of $\PP^2$ branched along a curve of degree $6$ is a $K3$ surface of degree $2$. For a given $X$ we can associate naturally two such surfaces. 
Those will be naturally the base of our moduli space of stable twisted sheaves.

Recall that the moduli space of stable twisted sheaves were described by Yoshioka in \cite{Yo}.
In order to construct such a moduli space $M_v(S,\alpha)=M(v)$ we need to fix a $K3$ surface $S$ with an element $\alpha$ in the Brauer group $Br(S)$ and a Mukai vector $v$.
Recall that for a $K3$ surface we have 
$$Br(S)=\Hom(T_S,\QQ/ \Z),$$
where 
$$T_S =NS(S)^{\perp} = \{ v\in H^2(S,\Z)\colon \ \forall \  m\in NS(S) \ \ \ v\cdot m=0 \}$$ is the transcendental lattice of $S$.
For a cyclic element $\alpha \in Br(S)$ of order $n$ denote by
$$T_{\langle\alpha\rangle}= ker(\alpha \colon T_S \to \QQ / \Z)\subset T_S$$ the sublattice of index $n$.

Now let $S$ be a general $K3$ surface of degree $2$ such that $NS(S)=\Z h$. 
 In \cite{vG} van Geemen classified the order $2$ elements in $Br(S)$ by classifying 
the possible index two sublattices of $T_S=\langle-2 \rangle \oplus 2 U\oplus 2 E_8(-1)$ and found three possibilities.
 Recall that an element of order $n$ in the group $Br(S)$ can be represented as a Brauer-Severi variety being a rank $n$ bundle on $S$. As suggested to us by A.~Kresch it is convenient to look at the three geometric realizations of the order two elements in $Br(S)$ in the following way:
\begin{itemize}
\item a $(2,2,2)$-complete intersection is the base locus of a net of quadric 4-folds,
then the space of planes in the quadric 4-folds is generically $\PP^3 \cup \PP^3$
and over the sextic discriminant curve is just $\PP^3$; Therefore it is a $\PP^3$-bundle
over the double cover of $\PP^2$ branched along the discriminant sextic
(an element of $Br(S)_2$ is also an element from $Br(S)_4$ so gives a rank $4$ bundle).

\item for the cubic fourfold containing a plane, the projection from the plane yields a
quadric surface fibration over $\PP^2$, the  discriminant locus is a sextic curve, the
space of lines in the quadrics give a $\PP^1$-bundle over the double cover branched along the sextic.

\item the double cover of $\PP^2 \times \PP^2$ branched along a $(2,2)$ hypersurface is a quadric surface bundle (by the projection to the first factor), the discriminant locus is
a sextic and the spaces of lines give a $\PP^1$-bundle over the corresponding double cover.
\end{itemize}

We are interested in the last case, discussed in detail in \cite[\S 9.8]{vG}. Then $$T_{\langle\alpha_3\rangle}=\langle -2 \rangle\oplus U\oplus U(2)\oplus 2 E_8(-1)$$ is Hodge isometric to a primitive sublattice 
of the middle cohomology of the Verra fourfold. Note that $T_{\langle\alpha_3\rangle}$ admits two embeddings as an index $2$ sublattice of $T_S$.  Note also that Hassett and Varilly-Alvarado \cite{HV} showed that the Brauer elements $\alpha_3$ that we consider can obstruct the Hasse principle.

\begin{prop}\label{twist} Let $X\in \mathcal{U}$ be general then $X$ is isomorphic to the moduli space of stable twisted sheaves on a $K3$ surface of degree 
$2$ with twist $\alpha_3\in Br(S)$.
\end{prop}
\begin{proof} 
Since the Picard group of $X$ has rank two and $X$ admits two Lagrangian fibrations it follows that the movable cone of $X$ is isomorphic to the 
nef cone. Thus it is enough to prove that $X$ is birational to the moduli space of twisted sheaves.

We argue similarly as \cite{Addington} using the global Torelli theorem \cite{Verbitsky}.
Let us use the notation from \cite[Proposition~4.1]{Huy}.
We have to show that
there is an embedding $H^2(X,\Z) \hookrightarrow \tilde{H}(S,\alpha_3)$  (into the Hodge structure of the twisted $K3$ surface see \cite[Definition~2.5]{Huy}) that is compatible with the Hodge structure.
Given the embedding, we find a vector $v\in \tilde{H}(S,\alpha_3)$ in the orthogonal complement of the image of $H^2(X,\Z)$ having $(v,v)=2$.
 For such $v$ let $M(v)=M_v(S,\alpha_3)$  be the moduli space of stable twisted sheaves on $S$. We know from \cite[Theorem~3.19]{Yo} that  there is a distinguished embedding
$$H^2(M(v),\Z)\simeq v^{\perp}\hookrightarrow \tilde{H}(S,\alpha_3).$$
 We deduce an isomorphism $H^2(X,\Z)\simeq H^2(M(v),\Z)$ and conclude by the global Torelli theorem
for deformations of $K3^{[2]}$ \cite[Corollary 9.8]{Markmansurvey} that  $M(v)$ and $X$ are birational. 

We denote the hyperbolic plane by $U$, i.e. $\Z^2$ with the intersection form 
$\left( \begin{array}{cc}
0 & 1 \\
1 & 0  \\
\end{array} \right) $.
Denote by $$\tilde{\Lambda}=I\oplus J\oplus M\oplus N \oplus 2E_8(-1),$$ where $I,J,M,N$ are copies of the hyperbolic lattice $U$.
We can assume (by choosing an appropriated marking) that $H^2(X,\Z)=\Lambda \oplus \langle-2\rangle=\eta^{\perp} \subset \tilde{\Lambda}$, where $\Lambda\simeq 3U\oplus E_8(-1)$ and $\eta$ is some element with $\eta^2=2$ contained in $I$. 
We know that the Hodge structure on $H^2(X,\Z)\otimes \C$ is determined by the choice of  $$x:=\langle H^{2,0}(X)\rangle$$ such that the algebraic part $H^{1,1}(X)$ of $x^{\bot}$ contains the 
lattice $U(2) \subset \tilde{\Lambda}$. We are thus in the context of \cite[Lemma~2.6]{Huy}.
From the improved Eichler's criterion  \cite[thm I.2.9]{BPV} it follows that there is a unique (up to $O(\tilde{\Lambda})$) embedding of a lattice of type $U(2)\oplus \langle\eta\rangle$  into $\tilde{\Lambda}$. 
In particular, we can assume that $U(2)\subset M\oplus N$ such that if $u_1,u_2$ and $v_1,v_2$ are standard generators of $M$ and $N$ respectively, then the image of the embedding is defined by $e=u_1+v_1, f_2=u_2+v_2$.
We find that the lattice generated by $\langle e,f\rangle$, where $f=u_2$ (or $f=v_2$) spans a hyperbolic plane $U$.
We obtain a new special decomposition \begin{equation}\label{decomp} \tilde{\Lambda}=U\oplus \Lambda\end{equation} 
 (this decomposition is different from the one in the definition of $\tilde{\Lambda}$).
Since $x$ is orthogonal to $e$,  it admits a decomposition $x=\lambda e+\sigma$ with respect to (\ref{decomp})   with $\sigma\in \Lambda \otimes \C$ and $\lambda \in \C$.

By the surjectivity of the period map we can find a $K3$ surface $S$ that realizes $\sigma\in \Lambda \otimes \C$ (we have two such $K3$ surfaces).
We claim that $S$ admits a polarization of degree $2$.
Indeed, observe that $\eta\in I$, so we have $\eta \in \Lambda$. Moreover, $$0=(\eta.x)=(\eta.(\lambda e+\sigma))=(\eta.\sigma)=(\eta.\bar{\sigma})$$ and $\eta^2=2$.
It follows that $\eta$ induces a polarisation of degree $2$ on $S$; the claim follows.

Let us identify the twist. As in \cite[Lemma~2.6]{Huy} we decompose with respect to (\ref{decomp})  the element $f_2=\gamma+2f+ke$ with $\gamma \in \Lambda$.
We compute that $\gamma=v_2-u_2 \in \Lambda$ and denote $B:=\frac{1}{2} \gamma=\frac{1}{2}(v_2-u_2)$. We define the Brauer class $\alpha_3' \in Br(S)$ as the image of
$B$ under the exponential map $$\Lambda\otimes \QQ\simeq H^2(S,\QQ) \to H^2(S,\mathcal{O}^*_S)_{tors}\simeq Br(S).$$ Finally, if we identify $\Z e\subset U\subset \Lambda\oplus U$ 
(with respect to (\ref{decomp})) 
with $H^4(S,\Z)$ we obtain an isometry with $\tilde{H}(S,\alpha_3')$ and the Hodge structure determined by $x$ on $\tilde{\Lambda}$ (and a second isometry for $ f=v_2$).

In order to identify the element $\alpha_3'$ with the element $\alpha_3$ described above, we use \cite[\S 2.1]{vG}.
Indeed, we associate to $B$ a map $b \colon T_S\to \QQ/\Z$, where $T_S\subset \Lambda$ is the perpendicular lattice to $\eta$ (in particular $T_S\simeq \langle-2\rangle \oplus 2 U \oplus 2 E_8(-1)=\langle-2\rangle\oplus \Lambda'$ and $B\in \Lambda'$), such that $b(t)=[t.B]$ (in particular $a_{\alpha_3'}=0$ and $d=1$ in the notation of \cite[Proposition~9.2]{vG}) i. e. we are in the case of \cite[Proposition~9.8]{vG}.
\end{proof}
\begin{rem} By analogy with the generic cubic fourfold containing a plane we expect that the generic Verra fourfold is not rational. We hope that the IHS fourfold from $\mathcal{U}$ related to a Verra fourfold can be used to attack this problem \cite{MS}. 
\end{rem}
\begin{rem} We saw in the proof above that $X$ admits two structures of moduli spaces of stable twisted sheaves on $K3$ surfaces of degree $2$.
The elements of these moduli spaces are torsion sheaves that are supported on curves on the linear system of degree $2$ on $S$, so define two Lagrangian fibrations.
\end{rem}
 
We are now ready to give a proof of our main Theorem.
\begin{proof}[Proof of Theorem \ref{main}]
It follows from Theorem \ref{dwa} and Proposition \ref{prop3.6} that the Lagrangian degeneracy locus $D_1^{\bar{A}}$ admits
a double cover branched along $D_2^{\bar{A}}$ being an IHS fourfold of type $K3^{[2]}$ such
that the double cover $X_{\bar{A}}\to D_1^{\bar{A}}$ is given by an anti-symplectic involution. Moreover, $X_{\bar{A}}$ 
moves in a $19$-dimensional family. It follows from \cite{OW} that the invariant lattice of the involution is one of the lattices $U$, $U(2)$, $\langle-2\rangle \oplus \langle 2\rangle$.
On the other hand $X_{\bar{A}}$ admits a polarisation of Beauville degree $4$ and, from Section \ref{lagr}, two Lagrangian fibrations: thus the invariant lattice is $U(2)$.

The fact that the Hilbert scheme of conics on $Y$ admits a $\PP^1$ fibration with base a fourfold from $\mathcal{U}$, follows from  Proposition \ref{Z_Q}.
The isomorphism with the moduli space of twisted sheaves follows from Proposition \ref{twist}.
\end{proof}

\section{ 
The Fano surface of the Verra threefold $Z$} \label{Verra-Involutions}


Let $Z= ( \PP_1^2\times \PP_2^2) \cap Q$ be a general Verra threefold, and let $F = F(Z)$ be the Fano surface 
of conics of bidegree $(1,1)$ on $Z$, 
i.e. 
$$
F = \{ [C]:C \subset Z\mbox{ is a conic},  C\cdot h_1 = C\cdot h_2 = 1 \}
$$
where the $h_i$ is the pullback to $Z$ of a line in $\PP_i^2$.
On $F$ there is a natural regular involution 
$$
i:F \rightarrow F, \; [C]  \mapsto [C'] = i([C])
$$ 
described as follows:   Since any $[C] \in F$ is of bidegree $(1,1)$ then 
$C$ can degenerate only to a pair $C_o = L+M$ of 
intersecting lines, one of bidegree $(1,0)$ 
and the other of bidegree $(0,1)$. 
Indeed if $C_o = 2L$ is a double line, then the bidegree 
${\rm deg}(C_o) = (2,0)$ or $(0,2)$, a contradiction.  
Let $p_i:Z\to \PP_i^2, i=1,2$, be the two projections.  Then, for any $[C] \in F$ the projections 
$$L_i = p_i(C) \subset \PP^2_i, i = 1,2$$
are lines, and  the conic $C$ lies on the  
smooth quadric surface 
$$
S_2 = L_1 \times L_2 \subset \PP_1^2\times \PP_2^2.
$$
Since $Z$ is a quadratic section $\PP_1^2\times \PP_2^2\cap Q$ and $C\subset Z$, 
then 
$$
S_2 \cap Z = S_2 \cap Q = C + C',
$$
where also $C'$ is a $(1,1)$-conic on $Z$.  It is bisecant to $C$.
The involution on $F$ is defined by $$i: [C]  \mapsto [C'].$$
 Clearly $[C] = i([C'])$, and $C'$ is the unique conic on $Z$ bisecant 
to $C$.   
The Fano surface $F = F(Z)$ of the general $Z$ is smooth, 
the involution $i:F \rightarrow F$ is regular and has no fixed 
points; in particular the {\it quotient Fano surface} 
$$
F_0 = F_0(Z) = F/i
$$ 
of $Z$ is smooth  \cite{Iliev}, \cite{DIM}.   
 Note that, by Proposition \ref{Z_Q}, the quotient Fano surface $F_0$ is isomorphic to the singular locus of the EPW quartic section 
 associated to the Verra fourfold $Y$ being the double cover of $\PP_1^2\times \PP_2^2$ branched in $Z$. 
 In particular $F_0$ is isomorphic to the fixed locus of an antisymplectic involution on an IHS fourfold of $K3$ type from the family $\mathcal{U}$.
\medskip

\subsection{ The two conic bundle structures on $Z$ and invariants of the Fano surface
}\cite{Vera}, \cite{Iliev}

Let $Z = ( \PP_1^2\times \PP_2^2) \cap Q$ be general. 
For a point $x \in \PP^2_i$ denote by 
$C_x = p_i^{-1}(x)$ the fiber of $p_i$ over $x$.
If $x \in \PP^2_1$ (resp. $x \in \PP^2_2$) 
then $C_x \subset Z$ is a conic of bidegree $(0,2)$
(resp. of bidegree $(2,0)$). For the general $Z$ any 
degenerate fiber $C_x$ of any of the two projections 
$p_i$ has rank two, i.e. $C_x = p_i^{-1}(x) = L'_x + L''_x$ 
is a pair of lines, intersecting at a point 
$$f_i(x) = L'_x \cap L''_x = \mbox{ Sing } C_x,$$  
and the discriminant curves 
$$
\Delta_i = \{ x \in \PP^2_i: C_x = L'_x + L''_x \} \subset \PP^2_i
$$
are smooth plane sextics, see \cite{Vera}.  
The maps 
$$
f_i: \Delta_i \rightarrow  \PP^8, \ 
x \mapsto f_i(x) = \mbox{ Sing } C_x, i = 1,2
$$
are called the 
the Steiner maps of the conic fibrations $p_i$.

Let 
$$\tilde{\Delta}_i = \{ ([L],x):L\subset C_x, x \in \Delta_i \}$$
be the curve of components of degenerate fibers $p_i^{-1}(x) =C_x= L'_x + L''_x$ 
of $p_i$, $i = 1,2$. Let 
$$
\pi_i: \tilde{\Delta}_i \rightarrow \Delta_i,\;  ([L],x)\mapsto x
$$
$i = 1,2$ be the induced \'etale double covering,
and let $\varepsilon_i \in Pic_2(\Delta_i)$ be the 
2-torsion sheaf defining $\pi_i$. Then for $i = 1,2$ 
the two coverings $\pi_i$ (resp. the two pairs $(\Delta_i,\varepsilon_i)$)
define two Prym varieties 
$$
Pr_i = {\rm Prym}(\Delta_i,\varepsilon_i)
$$ 
which are both principally polarized abelian varieties (p.p.a.v.) 
of dimension $9 = g(\Delta_i) - 1$. 
Let also 
$$
J(Z) = H^1(\Omega^2_Z)^*/H_3(Z,\Z)
$$
be the principally polarized (p.p.) intermediate jacobian of $Z$. 
By the results of \cite{Vera}, $Pr_1$ and $Pr_2$ are both 
isomorphic to each other and to $J(Z)$ as p.p.a.v..



\begin{prop}\label{3-2}
Let $Z = W \cap Q \subset \PP^8$ be a general Verra threefold. 
Then:  

(A) The Fano surface $F = F(Z)$ has invariants 
$K^2 = 576$, $c_2 = 312$, 
$p_g = 82$, $q = 9$.

(B) The quotient Fano surface $F_0 = F(Z)/i$ has invariants 
$K^2 = 288$, $c_2 = 156$, 
$p_g = 36$, $q = 0$.   
\end{prop}

\medskip 


\begin{proof}

The irregularity $q(F) = 9$ follows from the Abel-Jacobi isomorphism
$Alb(F) \cong J(Z)$, from where $q(F) = h^{1,0}(F) = h^{2,1}(Z) = 9$. 

We have seen that $F_0 $ is isomorphic to the fixed locus of the 
an IHS fourfold from $\mathcal{U}$. Starting from this, we compute 
the invariants of $F(Z)$ and $F_0$. 
The facts that  $K^2_{F_0} = 288$ and $\chi(\oo_{F_0})= 37$ follows from \cite{beau-invo} since $Y$ moves in a $19$ dimensional family. 
By Noether's formula, we infer $c_2(F_0)=156$.

Now, by \cite{Beauville3}, the class of $F$ in $J(Y)$
is $[F] = 2\Theta^7/7!$, where $\Theta$ is the theta divisor
of the principal polarization on $J(Y) $. Moreover from \cite[Corollary~3.17]{Vo1} the Abel-Jacobi map is surjective. Thus, by \cite[Corollary~3.18]{Vo1}, we deduce that the invariant part $H^0(K_{F(Z)})^+$ of the involution $i$ on $H^0(K_{F(Z)})$ has dimension $36$.
It follows that $p_g(F_0)=36$.

Since $F(Z) \rightarrow F_0$ is a 2-sheeted unbranched covering, then
$$
K^2_{F(Z)} =2 \cdot K^2_{F_0} = 576, $$
$$
c_2(F(Z)) = 2 \cdot c_2(F_0) = 2\cdot 156= 156,
$$
and $\chi(O_{F}) = 2\cdot \chi(O_F) = 37\cdot2 = 74$.  
Therefore, since $p_g(F_0) - q(F_0) + 1 =  \chi(O_{F_0})$ then 
$$
q(F_0) = p_g(F_0) - \chi(O_{F_0}) + 1 = 36- 37+ 1 = 0.
$$ 
Thus we find $q(F_0)=0$.
\end{proof}
\begin{rem} Note that the Chow group $CH^0(F(Z))$ was studied in \cite[Proposition~1.1]{Vo1}.
In particular, we find a description of the class $2([C]+[C'])\in CH^0(F(Z))$ where $C$ and $C'$ are two involutive conics.
\end{rem}
\begin{rem} Note that the Hilbert scheme of $(1,1)$ conics on $Z$ was already studied in \cite[\S 6]{Vera} and \cite{DIM}.
Verra considers a natural birational map $u\colon F(Z)\to D_Z^6$, where $D_Z^6$ is the intersection of the closure of the locus of rank $6$ quadrics containing $Z$ but not $\PP^2\times \PP^2$ with the locus of quadrics containing $\PP^2\times \PP^2$. 

Note that in \cite{DIM} it is proved that each nodal  prime Fano threefold $X_{10}$ of degree $10$ is birational to a Verra threefold $Z_X$.
From \cite[Proposition~6.6., \S 5.4]{DIM} the Fano surface of conics on $X_{10}$ and $(1,1)$ conics $Z_X$ are two birational  Beauville special subvarieties $S^{odd}$ and $S^{even}$ respectively.

 By \cite{Beauville3}, the class of $Z$ in $J(Z)$
is $[Z] = 2\Theta^7/7!$, where $\Theta$ is the theta divisor
of the principal polarization on $J(Z) $. 
Also $K_Z$ is numerically equivalent 
(on $Z$)  to $2\Theta |_Z$. Therefore we recompute
$$
K_Z^2 = (2\Theta|_Z)^2 = (2\Theta)^2.2\Theta^7/7! 
= 8\Theta^9/7! = 8.9!/7! = 8.8.9 = 576
$$
since the Abelian 9-fold $J(Z)$ is principally polarized by $\Theta$,
this yields $\Theta^9/9! = 1$. 
 
\end{rem}
\medskip

\begin{rem}
We also have the following relation $K_{F_0}=2 H+e$ where $H$ is the hyperplane section on $C(\PP^2\times \PP^2)\subset \PP^9$ and $e$ is the torsion divisor
defining the cover $i$.
\end{rem}

\bibliography{biblio} \bibliographystyle{alpha}
\end{document}